\documentclass[10pt,reqno]{amsart}
\usepackage{amsmath}
\usepackage{amssymb,amsfonts,mathrsfs, amsthm, stmaryrd, tabu, graphicx, lineno}
\usepackage[colorlinks=true,linkcolor=blue,citecolor=blue]{hyperref}
\usepackage[margin=2 cm,heightrounded=true,centering]{geometry}

\usepackage{amscd}
\usepackage{verbatim}
\usepackage{euscript}

\newtheorem{theorem}{Theorem}[section]
\newtheorem{prop}[theorem]{Proposition}
\newtheorem{lemma}[theorem]{Lemma}
\newtheorem{cor}[theorem]{Corollary}

\newtheorem{definition}[theorem]{Definition}
\newtheorem{remark}[theorem]{Remark}

\newcommand{\bR}{\mathbb{R}}

\newcommand{\vep}{\varepsilon}

\newcommand{\sW}{\mathcal{W}}

\newcommand{\calO}{\mathcal O}
\newcommand{\del}{\partial}
\newcommand{\RR}{\mathbb R}

\DeclareMathOperator{\tr}{tr}

\DeclareMathOperator{\dvol}{dV}
\DeclareMathOperator{\Rc}{R}
\DeclareMathOperator{\Ric}{Ric}
\renewcommand{\P}{\mathrm{P}} 
\DeclareMathOperator{\W}{W}
\DeclareMathOperator{\Ba}{Ba}
\DeclareMathOperator{\Scal}{S}
\DeclareMathOperator{\R}{R}

\newcommand{\gbar}{\overline{g}}

\newcommand{\riem}{\R}

\newcommand{\riemdddd}[4]{\riem_{#1 #2 #3 #4}}

\newcommand{\riemdudu}[4]{\riem^{\phantom{#1} #2 \phantom{#3} #4}_{#1 \phantom{#2} #3}}

\newcommand{\riemddud}[4]{\riem^{\phantom{#1 #2} #3}_{#1 #2 \phantom{#3} #4}}
\newcommand{\riemdddu}[4]{\riem^{\phantom{#1 #2 #3} #4}_{#1 #2 #3}}

\newcommand{\riemuudu}[4]{\riem_{\phantom{#1 #2} #3 \phantom{#4}}^{#1 #2 \phantom{#3} #4 }}

\newcommand{\riemuddu}[4]{\riem^{#1 \phantom{#2 #3} #4}_{\phantom{#1} #2 #3 \phantom{#4}}}

\newcommand{\inn}[2]{\left\langle #1 \, , \, #2 \right\rangle}

\newcommand{\ttv}{\stackrel{\circ}{v}}

\newcommand{\vdd}[2]{v_{#1 #2}}
\newcommand{\vdu}[2]{v_{#1}^{\phantom{#1} #2}}
\newcommand{\vddu}[3]{v_{#1#2,}^{\phantom{#1#2}#3}}
\newcommand{\vdud}[3]{v_{#1\phantom{#2},#3}^{\phantom{#1}#2}}
\newcommand{\vdddu}[4]{v_{#1 #2, #3}^{\phantom{#1#2,#3} #4}}
\newcommand{\vddud}[4]{v_{#1 #2, \phantom{#3} #4}^{\phantom{#1#2,} #3 \phantom{#4}}}
\newcommand{\vdudd}[4]{v_{#1 \phantom{#2}, #3 #4}^{\phantom{#1} #2 \phantom{, #3 #4}}}
\newcommand{\vdudu}[4]{v_{#1 \phantom{#2}, #3 \phantom{#4}}^{\phantom{#1} #2 \phantom{, #3} #4}}
\newcommand{\vdddudu}[6]{v_{#1#2,#3\phantom{#4}#5\phantom{#6}}^{\phantom{#1#2,#3}#4\phantom{#5}#6}}
\newcommand{\vddduud}[6]{v_{#1#2,#3\phantom{#4#5}#6}^{\phantom{#1#2,#3}#4#5\phantom{#6}}}
\newcommand{\vdduddu}[6]{v_{#1#2,\phantom{#3}#4#5\phantom{#6}}^{\phantom{#1#2,}#3\phantom{#4#5}#6}}
\newcommand{\vddudud}[6]{v_{#1#2,\phantom{#3}#4\phantom{#5}#6}^{\phantom{#1#2,}#3\phantom{#4}#5}}
\newcommand{\vdudddu}[6]{v_{#1\phantom{#2},#3#4#5\phantom{#6}}^{\phantom{#1}#2\phantom{,#3#4#5}#6}}
\newcommand{\vdddudd}[6]{v_{#1#2,#3\phantom{#4}#5#6}^{\phantom{#1#2,#3}#4\phantom{#5#6}}}
\newcommand{\vdduddd}[6]{v_{#1#2,\phantom{#3}#4#5#6}^{\phantom{#1#2,}#3\phantom{#4#5#6}}}
\newcommand{\vdudddd}[6]{v_{#1\phantom{#2},#3#4#5#6}^{\phantom{#1}#2\phantom{,#3#4#5#6}}}

\newcommand{\gendddudu}[7]{{#1}_{#2#3,#4\phantom{#5}#6\phantom{#7}}^{\phantom{#2#3,#4}#5\phantom{#6}#7}}
\newcommand{\gendddu}[5]{{#1}_{#2 #3, #4}^{\phantom{#2#3,#4} #5}}

\newcommand{\genddduu}[6]{{#1}_{#2 #3, #4}^{\phantom{#2#3,#4} #5#6}}
\newcommand{\genddudu}[6]{{#1}_{#2 #3,\phantom{#4} #5}^{\phantom{#2#3,}#4 \phantom{#5}#6}}

\newcommand{\fdudu}[4]{{f,}_{#1 \phantom{#2} #3 \phantom{#4}}^{\phantom{#1} #2 \phantom{#3} #4}}
\newcommand{\fdu}[2]{{f,}_{#1}^{\phantom{#1} #2}}

\newcommand{\varChr}[4]{ \frac{1}{2} g^{#3 #4} \left(  v_{#4 #2,#1} +  v_{#4 #1,#2} -  v_{#1 #2,#4} \right) }

\begin{document}

\title[Bach flow]{Linear and nonlinear stability for the Bach flow, I} 

\author[Bahuaud]{Eric Bahuaud}
\address{Department of Mathematics,
	Seattle University,
	Seattle, WA, 98122, USA}
\email{bahuaude [AT] seattleu [DOT] edu}
\author[Guenther]{Christine Guenther}
\address{Department of Mathematics,
	Pacific University,
	Forest Grove, OR 97116, USA}
\email{guenther [AT] pacificu [DOT] edu}
\author[Isenberg]{James Isenberg}
\address{Department of Mathematics,
	University of Oregon,
	Eugene, OR 97403-1222 USA}
\email{isenberg [AT] uoregon [DOT] edu}
\author[Mazzeo]{Rafe Mazzeo}
\address{Department of Mathematics, Stanford University, Stanford, CA}
\email{rmazzeo [AT] stanford [DOT] edu}

\date{\today}
\date{\today}

\begin{abstract}
In this paper we prove the linear stability of a gauge-modified version of the Bach flow on any complete manifold $(M, h)$ of
constant curvature. This involves some intricate calculations to obtain spectral bounds, and in particular introduces a higher order generalization
of the well-known Koiso identity.  We also prove nonlinear stability for the Bach flow if $(M, h)$ is the hyperbolic space $\mathbb H^n$, and more generally 
any Poincar\'e-Einstein space sufficiently close to $\mathbb H^n$.
In the forthcoming Part II of this project, we study the nonlinear stability question if $M$ is either compact or else noncompact and flat, since
those cases require different considerations involving a center manifold.
\end{abstract}

\maketitle

\section{Introduction}
The search for canonical metrics in Riemannian geometry has a long and rich history, starting from the classical uniformization theorem
for Riemann surfaces, with many significant milestones along the way.  Canonical metrics often arise as critical points
of an energy functional on some given space of metrics (e.g.,\ all metrics in a conformal class, or all metrics in a K\"ahler class, or
all metrics), and because of this, techniques from the calculus of variations have played a key role in many of these problems.
Particularly if the critical points are (local or global) minima of their associated energy functionals, gradient flow techniques 
provide another important path to finding such metrics.  Conversely, gradient flows, or modified versions of such flows, provide interesting
settings for the development of techniques to study various stability questions around canonical metrics. 

\

One source of natural energy functionals is provided by the Chern-Gauss-Bonnet theorem on a $4$-dimensional compact manifold $M$ \cite{Besse} which states that: 
\begin{equation*}
8 \pi^2 \chi(M) = \int_{M} \left(\frac{1}{4} |\W|^2 - \frac{1}{2} \left|\Ric - \frac{1}{4} \Scal g\right|^2 + \frac{1}{24} \Scal ^2 \right)\dvol_g.
\end{equation*}
Here $\W$, $\Ric - \frac{1}{4} \Scal g$ and $\Scal$ are the Weyl, traceless Ricci and scalar curvatures, and $|\W|$ is the norm of the $(0,4)$
tensor $\W$, all with respect to the metric $g$, and $\chi(M)$ is the Euler characteristic. It is well known that any quadratic curvature functional in 4 dimensions is a linear combination of the following three quantities:  $\int |\W|^2 \dvol_g $, $\int |\Ric|^2\dvol_g $, and  $\int \Scal^2\dvol_g $ 
(see e.g., \cite{Besse, GV}). One motivation for this paper comes from  the Weyl energy functional
\begin{align*}
g \longmapsto \sW(g) := \int_{M} | \W(g) |^2_g \dvol_g.
\end{align*}
Specifically in dimension $4$, this integral is unchanged if we replace $g$ by any conformally related metric.

\

The first variation of $\sW$ along a curve of metrics $g(s)$ with $g(0) = h$ and $g'(0) = v$ is given by
\begin{align*}
\left. \frac{d}{ds}  \sW( h + sv ) \right|_{s=0} = \int_{M} \inn{\nabla \sW}{ v}_h \dvol_h, 
\end{align*}
where 
\[ 
\nabla \sW = -4 \Ba(h)
\]
is a multiple of the Bach tensor of $h$, which is defined as follows: 
\begin{align} \label{eqn:def-bach}
\Ba_{ij} := {\P_{ij,k}}^k - {\P_{ik,j}}^k + \P^{kl} \W_{kijl}.
\end{align}
Here $\P:= \frac{1}{2} \left( \Ric - \frac{\Scal}{6} h\right)$ is the Schouten tensor. In this formula and throughout this manuscript, commas denote covariant 
derivatives, and the standard summation convention is used. The Bach tensor is traceless and divergence-free, and involves derivatives up to order four of 
the metric. The Schouten tensor $\P$ has a relatively simple conformal transformation rule,
and using this, the Bach tensor $\Ba$ has conformal invariance property:  
\[
g_1 := f^2 g \Longrightarrow \Ba(g_1) = f^{-2} \Ba(g).
\]
Metrics satisfying condition $\Ba = 0$ are called Bach flat; examples include all Einstein and conformally Einstein metrics. Bach flat 
metrics are expected to be somewhat easier to find than Einstein metrics. 

\

The Bach tensor defined by equation \eqref{eqn:def-bach} is extended to higher dimensions by the same formula with $P := \frac{1}{n-2} \left( \Ric - \frac{\Scal}{2(n-1)} h\right)$.  The Bach tensor loses several key properties if $n \geq 5$:
it is no longer conformally covariant, it is not the gradient of a Riemannian functional, and it is not divergence-free.  The Bach tensor 
has an analogue in higher dimensions which retains these properties \cite{Graham}. It is called the ``ambient obstruction tensor",  and it involves higher order derivatives of the metric.

\

In this paper we study a geometric flow associated to the Bach tensor.  This flow  decreases the Weyl energy functional in  dimension $4$.  Similar to the Ricci flow, this flow equation is quite degenerate
because of the underlying diffeomorphism invariance.  The first author and D.~Helliwell \cite{BH1} 
have introduced a modified system, hereafter called simply the ``Bach flow", which breaks this conformal invariance. The initial value problem for this Bach flow takes the form
\begin{equation}
\label{BFndim}
\frac{dg}{dt} = \Ba(g(t)) + \frac{1}{2(n-1)(n-2)} \Delta \Scal(g(t)) \, g(t), \; \; g(0) = g_0. 
\end{equation}
As explained in \cite{BGIM2}, stationary points of this modified flow are Bach flat metrics for which the scalar curvature is harmonic; this extra condition on $S$ provides a 
normalization within each conformal class. Below we also incorporate an analogue of the DeTurck gauge to counter the diffeomorphism invariance.  

\

Our interest in this paper is the long-time stability of the Bach flow, both at the linear and nonlinear levels, around a very special class of stationary points, namely the metrics 
with constant sectional curvature.  To this end, a significant part of the work below involves the explicit computation of the linearization of the gauge-adjusted Bach flow 
at metrics of constant sectional curvature and the establishment of bounds on the $L^2$ spectrum of the corresponding fourth order linear elliptic operator.   These spectral
bounds depend heavily on the choice of gauge fixing, and one of the important issues is to choose a gauge for which this linearization is self-adjoint. 
In our normalization, this $L^2$ spectrum is unbounded below. The flow is linearly stable if this spectrum is nonpositive, and one expects particularly nice convergence properties 
of the nonlinear flow if this spectrum is bounded above by a negative constant, or even if it is just nonpositive with a spectral gap at $0$.  If the spectrum is strictly
negative, we can employ analytic semigroup techniques as in \cite{BGIM} to conclude that the nonlinear flow converges to the constant curvature metric at an exponential
rate. Our first main result is the following: 
\begin{theorem}[Linear Stability of the Gauge-Adjusted Bach Flow at Spaces of Constant Sectional Curvature] \label{thm:main-1}
Let $(M^n, h)$ be an orientable complete Riemannian manifold with constant sectional curvature $c \in \{-1, 0, +1\}$.  There is a choice of gauge for which the 
linearization $L$ of \eqref{BFndim} at $h$ is self-adjoint, and for which the $L^2$ spectrum of $L$ is non-positive. More specifically: 
\begin{itemize}
\item If $c=0$ and $M$ is compact, the kernel of $L$ consists of the finite-dimensional space of parallel $2$-tensors. If $M$ is 
noncompact, we conclude only that the spectrum of $L$ is nonpositive, but there is no $L^2$ kernel. 
\item If $c = 1$ and $M = S^n$, the kernel of $L$ consists of the pure-trace symmetric $2$-tensors $v=f h$, where $f$ satisfies $\Delta_h f = -n f$. 
\item  If $c = -1$ and $M = \mathbb H^n$, the spectrum is bounded above by a negative constant; if $M$ is compact, however,
the kernel of $L$ consists of the symmetric $2$-tensors of the form $v = K\alpha = \delta^*_h \alpha$ where $\alpha$ is a harmonic $1$-form, along with the Codazzi-type 
tensors which are defined as tensors that satisfy $v_{ij,k} - v_{jk,i} = 0$. Here $K$ is the conformal killing operator (see \eqref{def:conformalKilling} for its definition).
\end{itemize}
In particular, if $M$ is compact, there is always a spectral gap at $0$. 
\end{theorem}

The choice of the gauge in Theorem \ref{thm:main-1} is different from the one proposed in \cite{BH1}; indeed, that one does not lead to a 
self-adjoint linearization.  The proof of Theorem \ref{thm:main-1} involves a careful estimation of $\langle Lv, v\rangle$. A decomposition of symmetric
$2$-tensors plays a central role. Similar to the linearization of the Einstein equation $\Ric(g) + (n-1) g = 0$ at a hyperbolic metric $h$, a special Bochner type formula 
is required to obtain the correct sign for this quantity if $v$ is tranverse-traceless. In the Einstein case, this Bochner formula is due to Koiso \cite{Koiso}.  The 
high-order analogue of the Koiso-Bochner formula proved here is new, of independent interest and likely to be useful in other geometric settings.

\

We next consider the nonlinear stability of the Bach flow at a constant curvature metric.   For reasons to be explained below, we focus solely on the special case of the nonlinear stability of hyperbolic space itself, and then deduce the analogous nonlinear stability result for Poincar\'e-Einstein metrics close to the hyperbolic using
perturbative methods.  For hyperbolic space there are no $L^2$ harmonic $1$-forms and no $L^2$ Codazzi-type tensors. We can thus prove that 
if the initial metric $g_0$ is sufficiently close to the hyperbolic metric (in the weighted little H\"older space ${\mathcal C^{k,\alpha}_{\mu}}$ defined in \eqref{wghtd norm def} in \S \ref{sec:dynamical}), 
then the corresponding solution of the nonlinear flow converges at an exponential rate to the hyperbolic metric.  Since this convergence result requires that $g_0-h$ decay at infinity,
this flow leaves the conformal infinity unchanged. 

\ 

\begin{theorem}[Nonlinear Stability of the Gauge-Adjusted Bach Flow at Hyperbolic Space]  \label{thm:main-2}  Let $(M,h)$ be $\mathbb H^n$ with its standard hyperbolic metric. For any $k\geq 4$, $0<\alpha<1$, and positive 
$\mu \in (\frac{n-1}{2} - r(n), \frac{n-1}{2}+r(n))$, where $r(n) > 0$ is a  constant depending on $n$, there exists $\vep > 0$ such that if  
\begin{equation*}
\| g_0 - h \|_{\mathcal C^{k,\alpha}_{\mu}} < \vep,
\end{equation*}
then the gauge-adjusted Bach flow $g_0(t)$ starting at $g_0(0) = g_0$ exists for all time and converges at an exponential rate to $h$ in the $\mathcal C^{k,\alpha}_{\mu}$ norm,
\begin{equation}\label{eqn:exp-decay}
\| g_0 (t) - h \|_{\mathcal C^{k,\alpha}_{\mu}} \leq C e^{-a(n) t}  \| g_0 - h \|_{ \mathcal C^{k,\alpha}_{\mu}},
\end{equation}
where $a(n)>0$ is a constant.
\end{theorem}
Using this result, we can then compensate for the gauge fixing using a family of diffeomorphisms and conclude the following:
\begin{cor}[Nonlinear Stability of the Bach Flow at Hyperbolic Space]  \label{cor:main-2}  If $g_0$ is a metric sufficiently close in $\mathcal C^{k,\alpha}_\mu$ norm to the hyperbolic metric $h$ on $\mathbb H^n$
for some positive $\mu \in (\frac{n-1}{2} -r(n), \frac{n-1}{2}+r(n))$, then the solution of the Bach flow $g_0(t)$ starting at $g_0$ exists 
for all time and converges at an exponential rate in $\mathcal C^{k,\alpha}_\mu$ to some pullback $F^*h$.
\end{cor}

\

We show in Appendix \ref{sec:appendix} that $r(n) =\frac{n-1}{8}$ suffices, and that the associated rate of convergence is $a(n)= \frac{(n-2)(3n-11)(5n-13)}{128}$ (see Remark \ref{rem:ind-poly}).  However, these constants are definitely not sharp.

\

Theorem \ref{thm:main-2} directly implies dynamical stability of Poincar\'e-Einstein spaces sufficiently close to hyperbolic space. We describe this condition 
more precisely in \S \ref{sec:spec bds PE}, but the result is the following: 

\begin{theorem}[Nonlinear Stability of the Gauge-Adjusted Bach Flow at Poincar\'e-Einstein spaces]
\label{thm: main P-E} 
For any $k\geq 4$, $0 < \alpha < 1$ and $\mu \in (\frac{n-1}{2} - r(n), \frac{n-1}{2}+r(n))$, where $r(n)>0$ is a constant depending on $n$, there exists $\vep > 0$ such that if  
$\hat{g}$ is a Poincar\'e-Einstein metric sufficiently close to $h$, 
and if
\begin{equation*}
\| g_0 - \hat{g} \|_{\mathcal C^{k,\alpha}_{\mu}} < \vep,
\end{equation*}
then the gauge-adjusted Bach flow $g_0(t)$ satisfying $g_0(0) = g_0$ exists for all time and converges to $\hat{g}$ in the $\mathcal C^{k,\alpha}_{\mu}$ norm; indeed, there exist constants $C>0$ and $a(n)>0$
such that 
\begin{equation} \label{eqn:exp-decay-pe}
\| g_0 (t) - \hat{g} \|_{\mathcal C^{k,\alpha}_{\mu}} \leq C e^{-a(n) t} \| g_0 - \hat{g} \|_{\mathcal C^{k,\alpha}_{\mu}}.
\end{equation}
\end{theorem}

\

This implies the nonlinear stability of the ungauged Bach flow in the usual way:

\

\begin{cor}[Nonlinear Stability of the Bach Flow at Poincar\'e-Einstein spaces] 
\label{cor:main P-E}
If $\hat{g}$ is a Poincar\'e-Einstein metric that is sufficiently close to $h$, and if \[\|g_0 - \hat{g} \|_{\mathcal C^{k,\alpha}_{\mu}} < \vep,\]
then the solution of the Bach flow $g_0(t)$ starting at $g_0$ exists for all time and converges at an exponential rate in $\mathcal C^{k,\alpha}_\mu$ to $\hat{g}$. 
\end{cor}

Theorem \ref{thm:main-2} addresses nonlinear stability in only one of the several cases discussed in the linear stability result, Theorem \ref{thm:main-1}.  Further nonlinear stability
results addressing those other cases will be handled in a sequel to this paper.  There are two somewhat distinct reasons for this. First, using techniques 
due to Koch and Lamm \cite{KL}, it is possible to show that the nonlinear (gauge fixed) flow exists for all time if $M$ is any complete 
noncompact flat manifold.  However, those techniques do not lead immediately to a convergence result, and in any case are of a very different flavor 
than the results in the present paper.  On the other hand, if $M$ is compact and $L$ has a nontrivial kernel (this occurs in all three cases; for $c = 1, 0,$ and $-1$), one expects the flow to have a center manifold, and one must show that the flow converges to some possibly different Bach flat metric 
contained in this center manifold. The techniques to do so are again quite different than the ones that we use here.   Both of these issues will be handled in this sequel.  
The papers \cite{BGIM2, BGIM} already contain relevant information and results about these nonlinear flows. 

\

We begin in \S \ref{sec:background} by discussing properties of the Bach tensor, outlining our conventions and stating variational formul\ae\
that are used below. Also in \S \ref{sec:background} we recall a standard decomposition of the space of symmetric $2$-tensors. 
We then define the gauge-modified Bach flow. The gauge term in that flow has parameters that we show can be chosen so that the linearization is self-adjoint. We then compute, in \S \ref{sec:linearize},
this self-adjoint linearization at metrics of constant sectional curvature, which are stationary points for the flow. In \S \ref{sec:l2spec} we estimate the $L^2$ spectrum for the linearized flow at these constant curvature metrics.  The proof of Theorem \ref{thm:main-1} is then split over the proof of Theorems \ref{thm:main1a}, \ref{thm:main1b}, \ref{thm:main1c} and \ref{thm:main1nc}.  In
\S \ref{sec:dynamical} we prove  Theorem \ref{thm:main-2}. The stated results in Theorem \ref{thm: main P-E} and Corollaries \ref{cor:main-2} and \ref{cor:main P-E} follow immediately. Appendix \ref{sec:appendix} contains
the calculations of the indicial roots of this linearization at the boundary of hyperbolic space; Appendix \ref{sec:techappendix} presents 
derivations of a few additional technical formul\ae.

\subsection{Acknowledgments}
This work was supported by  grants from the Simons Foundation (Simons Collaboration Grant \#426628, E. Bahuaud, and an AMS-Simons Research Enhancement Grant for PUI Faculty, C. Guenther).

\section{Background} \label{sec:background}

\subsection{Conventions} \label{sec:conventions}
In this section we gather the conventions used throughout the paper. First, $(M^n,h)$ is always a smooth orientable Riemannian manifold of constant sectional
curvature $c \in \{-1, 0, +1\}$.  The sign conventions for the curvature and the various linear operators below are consistent with those in the book \cite{CLN},
but the reader should beware that these are at odds with other standard conventions.  In particular, the Hodge Laplacians appearing below have nonpositive
spectrum. 

\

The full Riemannian, Ricci and Scalar curvature tensors of the constant curvature metric $h$ are 
\begin{align} \label{eqn:constcurvature}
\riemdddd{i}{j}{k}{l} &= c ( h_{il} h_{jk} - h_{ik} h_{jl} ), \; \; \;  \\
\Rc_{jk} &= h^{il}\riemdddd{i}{j}{k}{l} = c (n -1 ) h_{jk},\\
\Scal &= h^{jk} \Rc_{jk} = c n (n-1).
\end{align}
For constant curvature metrics, the Schouten tensor reduces to the much simpler expression
\begin{align} \label{eqn:schouten const curv}%\label{eqn:def-schouten}
\P_{ij} = \frac{1}{n-2} \left( \Rc_{ij} - \frac{\Scal}{2(n-1)} g_{ij}\right) = \frac{c}{2} h_{ij}.
\end{align}
and likewise the Weyl curvature 
\begin{align*}
\W_{ijkl} := \riem_{ijkl} - ( \P \owedge g)_{ijkl},
\end{align*}
vanishes if $g = h$.  Here $\owedge$ is the Kulkarni-Nomizu product (if $a$ and $b$ are symmetric $2$-tensors, then 
$(a \owedge b)_{ijkl} := a_{il} b_{jk} + a_{jk} b_{il} - a_{ik} b_{jl} - a_{jl} b_{ik}$). 

\ 

We use the convention for the scalar Laplacian
\begin{align*}
\Delta^g f := g^{jk} f_{,jk},
\end{align*}
with a similar sign choice for Laplacians on other tensor bundles, including the bundles of symmetric $2$-tensors $\Sigma^2( T^*M )$, of traceless symmetric $2$-tensors 
$\Sigma^2_0(T^* M)$, and of $p$-forms $\Lambda^p( T^*M)$.  The space of sections of a bundle $E$ is denoted by $\Gamma(E)$, and in particular,
$\Gamma(\Lambda^p(T^*M)) = \Omega^p(M)$. 
%The $L^2$ pairing induced by the metric is denoted by $( \cdot, \cdot )_g;$ for example, if applied to symmetric $2$-tensors $u,v$ this is
%\begin{align}
%(u, v)_g := \int_{M} \inn{u}{v}_g \; \dvol_g,
%\end{align}
%where $\inn{u}{v}_g$ denotes tensor contraction using the metric $g$. We define $L^2( M,g )$ as the completion of compactly supported smooth functions (or tensors) with respect to $\| u \| := ( u, u )_g^{1/2}$.
The divergence $\delta$ and its adjoint $\delta^*$, 
\begin{align*}
\delta : \, & \Gamma( \Sigma^2( T^*M ) ) \longrightarrow \Gamma( T^*M  ), & [\delta z]_{i} = - {z_{ij,}}^j   \\ 
\delta^* :\, &\Gamma( T^*M ) \longrightarrow \Gamma( \Sigma^2(T^*M) ),  & [\delta^* z]_{ij} = \frac{1}{2} ( z_{i,j} + z_{j,i} )
\end{align*}
appear frequently in our calculations. 

Again, contrary to a commonly used choice of sign, but following \cite{CLN}, we use
\begin{align*}
\Delta_H = -(d d^* + d^* d)
\end{align*}
as the Hodge Laplacian on differential forms, where $d:\Omega^p(T^*M) \rightarrow \Omega^{p+1}(T^*M)$ is the exterior derivative and $d^*: \Omega^p(T^*M)  \rightarrow \Omega^{p-1}(T^*M)$ is its $L^2$ adjoint. 

\subsection{Variational formul\ae }
We next record the various expressions for infinitesimal variations of the various geometric quantities and operators relative to a smooth 1-parameter family of metrics $g(s)$,
where $g(0) = h$ is the constant curvature metric and $g'(0) = v$.  First, 
\begin{align} 
\frac{d}{ds} g^{ij} &= -g^{ia} g^{jb} v_{ab}, \\
\frac{d}{ds} \Gamma_{ij}^k &=  \varChr{i}{j}{k}{m}, \label{eqn:var-christoffel} \\
\frac{d}{ds} \riemdddu{i}{j}{k}{l} &= \frac{1}{2} g^{lp} \left( v_{kp,ji} + v_{jp,ki} - v_{jk,pi} - v_{kp,ij} - v_{ip,kj} + v_{ik,pj} \right),\\
\frac{d}{ds} \Rc_{ij}  &= \frac{1}{2} g^{pq} \left(  v_{ip,jq} - v_{qp,ji} +  v_{jp,iq} -  v_{ij,pq} \right), \label{eqn:varRicci} \\
\frac{d}{ds} \Scal &= - {(g^{jk} v_{jk})_{,m}}^m + {v_{jk,}}^{kj} - v^{jk} \Rc_{jk}. \label{eqn:var-scal-curv}
\end{align}
Covariant derivatives are defined here with respect to $g(s)$ (hence, if evaluated at $s=0$, with respect to $h$). 

\

It is convenient to use expressions for variations of the Ricci and Schouten tensors obtained by straightforward commutation of covariant derivatives: 
\begin{prop} \label{prop:var-basic-curv}
Relative to a family of metrics $g(s)$ with $g(0) = h$, we have
\begin{align*}
	\partial_s \Rc_{ij} |_{s=0}	&= -\frac{1}{2} \vdddu{i}{j}{p}{p}  + \frac{1}{2} \vddud{i}{p}{p}{j}  + \frac{1}{2} \vddud{j}{p}{p}{i}  - \frac{1}{2} \vdudd{p}{p}{j}{i} - c \vdu{p}{p} h_{ij} + c n v_{ij},\\
	\partial_s \P_{ij}|_{s=0} &= -\frac{1}{2(n-2)} \vdddu{i}{j}{p}{p} + \frac{1}{2(n-2)} \vddud{i}{p}{p}{j}  + \frac{1}{2(n-2)} \vddud{j}{p}{p}{i} - \frac{1}{2(n-2)} \vdudd{p}{p}{j}{i}  \nonumber \\
	&+\frac{1}{2(n-1)(n-2)} \vdudu{p}{p}{m}{m} h_{ij} - \frac{1}{2(n-1)(n-2)} \vdudu{p}{m}{m}{p} h_{ij} \\
	&- \frac{c}{2(n-2)} \vdu{p}{p} h_{ij} + \frac{cn}{2(n-2)} \, v_{ij},\nonumber \\
	\left. \left( h^{ab} \partial_s \P_{ab} \right) \right|_{s=0}  \, h_{ij}
	&= \frac{1}{2(n-1)} \vdudu{p}{m}{m}{p}  h_{ij} - \frac{1}{2(n-1)} \vdudu{p}{p}{m}{m} h_{ij}. 
	\end{align*}
\end{prop}

\subsection{The Bach flow}
We define the Bach tensor $\Ba$ as in \eqref{eqn:def-bach}, and the Bach flow by \eqref{BFndim}. Noting that $\Ba$ is trace-free in all dimensions, there are three immediate consequences. 

First, if $n=4$, then $\sW(g)$ decreases under the (modified) Bach flow, even though this is not the true gradient flow. Indeed
\begin{align*}
\frac{d}{dt} \sW(g(t)) = \int_{M} \inn{ \nabla \sW}{g'(t)} \dvol_g = \int_{M} \inn{ -4 \Ba}{\Ba + \frac{1}{12} (\Delta \Scal) g} \dvol_g = -4 \| \Ba \|^2 < 0.
\end{align*}
(This fails in higher dimensions since $\Ba$ does not arise as the variation for $\sW$.) 
Second, the Bach flow preserves volume if $M$ is compact: 
\begin{align} \label{eqn:var-vol}
\frac{d}{dt} \mathrm{vol}(g(t)) = \int_{M} \frac{1}{2} \tr^g g'(t) \dvol_g = \int_{M} \left(\frac12 \tr^g \Ba + \frac{n}{4(n-1)(n-2)} \Delta S\right) \, dV_g = 0.
\end{align}
Finally, if $g$ is a fixed point of the flow and $M$ is compact, then taking the trace of $\Ba(g) + \frac{1}{2(n-1)(n-2)} \Delta \Scal g = 0$ yields that
$\Delta \Scal = 0$; hence, $\Scal$ is constant.  Consequently the fixed point equation is simply $\Ba(g) = 0$.

\

The contracted Bianchi identity yields
\begin{align} \label{eqn:def-mbach}
\Ba_{ij} + \frac{ \Delta \Scal}{2(n-1)(n-2)} g_{ij} &= \frac{1}{n-2} {\Rc_{ij,k}}^k   - {\P_{ik,j}}^k + \P^{kl} \W_{kijl},
\end{align}
and to compute the divergence of the Schouten tensor here, we commute derivatives and apply the contracted Bianchi identity again to get
\begin{align*} 
{\P_{ik,j}}^k = {{\P_{ik,}}^k}_j + {\R_{jki}}^m {\P_{m}}^k + {\Rc_{j}}^m \P_{im} %\nonumber \\
= \frac{1}{2(n-1)} \Scal_{,ij} + {\R_{jki}}^m {\P_{m}}^k + {\Rc_{j}}^m \P_{im}.
\end{align*}

Using this and introducing the Weyl tensor, we obtain 
\begin{align} \label{eqn:def-mbach-2}
\begin{aligned}
\Ba_{ij} &+  \frac{ \Delta \Scal}{2(n-1)(n-2)} g_{ij} \\
&= \frac{1}{n-2}{\Rc_{ij,k}}^k - \frac{1}{2(n-1)} \Scal_{,ij} - {\R_{jki}}^m {\P_{m}}^k - {\Rc_{j}}^m \P_{im} + \P^{kl} \W_{kijl} \\
&= \frac{1}{n-2} {\Rc_{ij,k}}^k  - \frac{1}{2(n-1)}  \Scal_{,ij} - 2 {\R_{jki}}^m {\P_{m}}^k - {\Rc_{j}}^m \P_{im} - \P^{kl} (\P \owedge g)_{kijl}. 
\end{aligned}
\end{align}
as the right hand side of the modified Bach flow. 

\subsection{Gauge adjustment}
\label{subsec:gauge_adj}

The flow generated by \eqref{eqn:def-mbach-2} is not parabolic as a consequence of its invariance under diffeomorphisms. Motivated by the DeTurck adjustment to the Ricci flow, we 
add a higher-order variant of the Bianchi gauge (see \eqref{def-Bianchi}) to obtain a shadow parabolic flow. To motivate this further, we briefly review the abstract Ricci-DeTurck 
setup discussed as in \cite{BH1}:  We recall that $T(g)$ is a natural tensor for a metric $g$ if, for any diffeomorphism $\phi$ of $M$, $T(\phi^* g) = \phi^*T(g)$. Now suppose $Z$ is a 
time-dependent vector field, and $g(t)$ satisfies the equation  
\begin{align}
\label{eqn:RcDTSetup}
\frac{dg}{dt} = T(g(t)) + \mathcal{L}_{Z(t)} g(t),
\end{align}
where $\mathcal{L}_{Z(t)}$ denotes the Lie derivative along the vector field $Z(t).$ We observe that if the flow by $T(g)$ preserves the volume of a compact manifold, then the 
gauge-adjusted flow \eqref{eqn:RcDTSetup} also preserves the volume since it differs by a pure divergence term.

If $-Z$ generates the one-parameter family of diffeomorphisms $\theta_t$ and $g$ satisfies equation \eqref{eqn:RcDTSetup}, then $\gbar = \theta_t^* g$ satisfies 
\begin{align} \label{eqn:abs-eqn}
\frac{d\gbar}{dt} = T(\gbar).
\end{align}
We recall as well that $[\mathcal{L}_Z g]_{ab} = \nabla_a \omega_b + \nabla_b \omega_a = 2 (\delta^* \omega)_{ab}$, where
$\omega$ is the $1$-form metrically dual to $Z$. In short, we obtain a gauge-modified flow
\begin{align} \label{eqn:abs-mod-eqn}
\frac{d g}{dt} = T(g(t)) + 2 \, \delta_{g(t)}^* \omega(t),
\end{align}
for any choice of a $t$-dependent $1$-form $\omega$.  Our goal is to find a propitious choice of $\omega$ which results in the linearized flow taking a particularly simple form.  Solutions of the gauge-modified flow are then immediately converted back to solutions of the original flow using the family of
diffeomorphisms obtained by integrating the associated vector field. 

\

We define the Bianchi operator as follows: 
\begin{align}
\label{def-Bianchi}
[\beta_h(g)]_i := -h^{jk} g_{ij,k} + (1/2) (h^{jk} g_{jk})_{,i}.
\end{align}
The contracted Bianchi identity for $h$ can be written as $\beta_h( \Ric^h ) = 0$.   For the $1$-form in the gauge adjustment setup above, we use
a linear combination of $\delta_h(g)$ and its Laplacian. Thus, for constants $\mu$ and $\nu$ to be determined below, we consider the following gauge-adjustment 
of \eqref{eqn:def-mbach-2}:
\begin{definition}(General gauge-adjusted Bach flow) 
\begin{align}
\label{eqn:def-mbach-3}
\nonumber \frac{ dg}{dt} &= \frac{1}{n-2} \left(\Delta_h R_{ij}  + 2 \left[ \delta_{g}^* \left( \frac{1}{2} \Delta_h \beta_h (g) + \frac{n-2}{4(n-1)} d\Scal  + \mu \, \delta_h g + \nu \, d( \tr^h g )  \right) \right]_{ij} \right) \\
&  - \frac{1}{2(n-1)}  \Scal_{,ij} - 2 {\R_{jki}}^m {\P_{m}}^k - {\Rc_{j}}^m \P_{im} - \P^{kl} (\P \owedge g)_{kijl}. 
\end{align}
\end{definition}

\ 

Let us motivate these new terms. 
As we show in \S \ref{sec:linearize}, the linearization of the sum of the terms involving $\Delta_h \beta_h (g)$ and $\Delta_h \Ric$ is elliptic. Next, $dS$ is included to 
cancel the term involving the Hessian of the scalar curvature term on the second line of the flow equation: 
\[
\frac{1}{n-2} \,2 \, \left[ \delta_g^* \frac{n-2}{4(n-1)} d\Scal \right]_{ij} - \frac{1}{2(n-1)} \Scal_{,ij} = 0. 
\]  
The final two terms in square brackets in equation \eqref{eqn:def-mbach-3} are lower-order in the derivatives of the metric, but are chosen specifically because if $\mu = -c(n-1)/2$ and $\nu = -c/4$,  where $c$ is the constant sectional curvature of $h$, then the linearization is self-adjoint.  This leads us finally to the following definition: 
\begin{definition}(Gauge-adjusted Bach flow)
\begin{align}
\label{eqn:def-mbach-3.5} 
\nonumber \frac{ dg_{ij}}{dt} = \frac{1}{n-2} \left( \phantom{\frac{1}{1}} {\Rc_{ij,k}}^k  \right. & \left. + \; 2 \left[ \delta_{g}^* \left( \frac{1}{2} \Delta_h \beta_h (g)  -\frac{c(n-1)}{2} \delta_h g - \frac{c}{4} d( \tr^h g ) \right) \right]_{ij} \phantom{\frac{1}{1}} \right) \\
& - 2 {\R_{jki}}^m {\P_{m}}^k - {\Rc_{j}}^m \P_{im} - \P^{kl} (\P \owedge g)_{kijl}. 
\end{align}
\end{definition}

\

Removing the $2 \delta_g^*( \cdot )$ term in \eqref{eqn:def-mbach-3} yields the Bach flow \eqref{BFndim}; hence, solutions to equation \eqref{eqn:def-mbach-3.5} can be converted by
pullback to solutions to the original Bach flow. 

\subsection{A splitting of the bundle of symmetric $2$-tensors} \label{sec:split}
We conclude this section by recalling a well-known decomposition of the space of symmetric $2$-tensors.  Let $K$ be the conformal Killing operator with respect to $h$:
\begin{align} \label{def:conformalKilling}
K: \Gamma( T^*M ) &\longrightarrow \Gamma( \Sigma^2_0( T^*M ) ), \nonumber \\
[K \alpha]_{ij} &:= \frac{1}{2} \left( \alpha_{i,j} + \alpha_{j,i} \right) - \frac{1}{n} {a_{k,}}^k h_{ij}.
\end{align}
There is an $L^2$ orthogonal decomposition  as follows:
\begin{align} \label{eqn:split}
\Gamma( \Sigma^2( T^*M ) )= \mathrm{im}(K) \oplus \{ f \cdot h \} \oplus  \{ v: \delta v =0, \tr^h v = 0 \}.
\end{align}
Tensors in the third summand are ``tranverse-traceless'',  which we write as TT. Though not appearing in this decomposition, $1$-forms in the kernel of $K$ are 
called conformal Killing $1$-forms.  A proof of this splitting on compact manifolds can be found in  \cite[Lemma 3.2] {ViaclovskyPCMI}.  Since this is also
needed here on asymptotically hyperbolic manifolds, we remark that the essentially the same proof applies. Indeed, the symbol of $K$ is injective,  hence
$K^* K$ is elliptic. Note too that $K^* = \delta$ on tracefree symmetric $2$-tensors.  After commuting derivatives (see for example equation \eqref{111} below) 
one sees that $\| K \alpha \|^2 \geq (n-1) \| \alpha\|^2$ for all $L^2$ $1$-forms $\alpha$ on hyperbolic space, hence $K$ has closed range, and thus by 
standard arguments, the decomposition \eqref{eqn:split} is valid as an orthogonal decomposition on $\mathbb H^n$.  If $M$ is only asymptotically hyperbolic and conformally
compact, this inequality holds for forms supported outside a sufficiently large compact set, which implies that $K^* K$ is Fredholm, and then the same
formal argument leads to the decomposition $L^2(M; ( \Sigma_0^2( T^*M ) ) = \mathrm{im}(K) \oplus \mathrm{ker} (K^*).$

\section{The linearization} \label{sec:linearize}
In this section we compute the linearization of the right-hand side of the general gauge-adjusted flow \eqref{eqn:def-mbach-3} at a metric $h$ of constant 
sectional curvature $c$.  The main results of this section are Theorem \ref{thm:Lin}, which documents the linearization, and Proposition \ref{prop:L-selfadj}, which 
specifies the choices of constants $\mu$ and $\nu$ that ensure that the linearization is self-adjoint. We relegate some lengthy but straightforward proofs of 
various commutation formul\ae\ to Appendix \ref{sec:techappendix}.

\

We begin by computing the linearization of the Laplacian of the Ricci tensor.
\begin{prop}
If $g(s)$ is a 1-parameter family of metrics with $h = g(0)$ a metric with constant sectional curvature $c$ and $v = g'(0)$, then 
\begin{align} \label{eqn:lin-tensor}
\begin{aligned}
& D[ \Delta^g \Ric(g) ]_h v_{ij} := \left. \frac{\partial}{\partial s} ( \Delta^{g(s)} \Ric(g(s))_{ij} )\right|_{s=0} \\
&\ \ = -\frac{1}{2} \vdddudu{i}{j}{p}{p}{m}{m}
  + \frac{1}{2} \vdduddu{i}{p}{p}{j}{m}{m} + \frac{1}{2} \vdduddu{j}{p}{p}{i}{m}{m} - \frac{1}{2} \vdudddu{p}{p}{j}{i}{m}{m} - c \vdudu{p}{p}{m}{m} h_{ij} + 
c \vdddu{i}{j}{m}{m}. 
\end{aligned}
\end{align}
\end{prop}
\begin{proof}
Differentiating 
\begin{align*}
\Rc_{ij,k} &= \nabla_k \Rc_{ij} = \partial_k \Rc_{ij} - \Gamma_{ki}^p \Rc_{pj} - \Gamma_{kj}^p \Rc_{ip}
\end{align*}
with respect to $s$ produces
\begin{align} \label{eqn:loc-a1}
\begin{aligned}
\partial_s \Rc_{ij,k} 
&= \partial_s \partial_k \Rc_{ij} - \Gamma_{ki}^p \partial_s \Rc_{pj} - \Gamma_{kj}^p \partial_s \Rc_{pi} - (\partial_s\Gamma)_{ki}^p \Rc_{pj} - \partial_s(\Gamma_{kj}^p) \Rc_{pi} \\
	&= \nabla_k (\partial_s \Rc_{ij}) - (\partial_s\Gamma)_{ki}^p \Rc_{pj} - \partial_s(\Gamma_{kj}^p) \Rc_{pi}.
\end{aligned}
\end{align}
Setting $s=0$ and applying Proposition \ref{prop:var-basic-curv} and equation \eqref{eqn:var-christoffel}, we obtain
\begin{align*}
\partial_s [\Rc_{ij,k}] |_{s=0}
&= \nabla_k \left( -\frac{1}{2} \vdddu{i}{j}{p}{p}  + \frac{1}{2} \vddud{i}{p}{p}{j} + \frac{1}{2} \vddud{j}{p}{p}{i} - \frac{1}{2} \vdudd{p}{p}{j}{i} - c \vdu{p}{p} h_{ij} + c n v_{ij} \right) - c(n-1) v_{ij,k} \nonumber \\
&= -\frac{1}{2} {\vdddu{i}{j}{p}{p}}_k  + \frac{1}{2} \vddud{i}{p}{p}{jk} + \frac{1}{2} \vddud{j}{p}{p}{ik} - \frac{1}{2} \vdudd{p}{p}{j}{ik} - c {\vdu{p}{p}}_{,k} h_{ij} + c v_{ij,k}.
\end{align*}

\

Taking another covariant derivative, 
\begin{align*}
\Rc_{ij,km} &= \partial_m \Rc_{ij,k} - \Gamma_{mi}^p \Rc_{pj,k} - \Gamma_{mj}^p \Rc_{ip,k} - \Gamma_{mk}^p \Rc_{ij,p},
\end{align*}
and calculating as has been done for equation \eqref{eqn:loc-a1}, we obtain
\begin{align*}
\partial_s \Rc_{ij,km} &= \nabla_m \partial_s \Rc_{ij,k} - \partial_s \Gamma_{mi}^p \Rc_{pj,k} - \partial_s \Gamma_{mj}^p \Rc_{ip,k} - \partial_s \Gamma_{mk}^p \Rc_{ij,p}.
\end{align*}
Therefore, at $s=0$, and using that $\nabla \Rc_{ij} = 0$ since $h$ has constant curvature, 
\begin{align*}
\partial_s [\Rc_{ij,km}] |_{s=0}
&= -\frac{1}{2} \vdddudd{i}{j}{p}{p}{k}{m}  + \frac{1}{2} \vdduddd{i}{p}{p}{j}{k}{m} + \frac{1}{2} \vdduddd{j}{p}{p}{i}{k}{m} - \frac{1}{2} \vdudddd{p}{p}{j}{i}{k}{m} - c \vdudd{p}{p}{k}{m} h_{ij} + c v_{ij,km}.
\end{align*}

\

Finally, differentiating $\Delta \Rc_{ij} = g^{km} \Rc_{ij,km}$ at $s=0$, and using that $\nabla \Rc_{ij} = 0$,  we see finally that
\begin{align*}
\partial_s [g^{km} \Rc_{ij,km}] |_{s=0}
&= -\frac{1}{2} \vdddudu{i}{j}{p}{p}{m}{m}  + \frac{1}{2} \vdduddu{i}{p}{p}{j}{m}{m} + \frac{1}{2} \vdduddu{j}{p}{p}{i}{m}{m} - \frac{1}{2} \vdudddu{p}{p}{j}{i}{m}{m}  - c \vdudu{p}{p}{m}{m}  h_{ij} + c \vdddu{i}{j}{m}{m}.
\end{align*}
\end{proof}

\begin{prop}
\label{prop:lin-high-order}
Using the same notation as above, we obtain:
\begin{align} \label{eqn:lin-gauge}
\begin{aligned}
D[ \delta_{g}^* & ( \Delta_h \beta_h (g)  ) ]_h v_{ij} = -\frac{1}{2} \vdduddu{i}{p}{p}{j}{m}{m} -\frac{1}{2} \vdduddu{j}{p}{p}{i}{m}{m} + 
\frac{1}{2} \vdudddu{p}{p}{i}{j}{m}{m} \\ & - 2c  \vdudu{m}{p}{p}{m} h_{ij} + c \vdudu{p}{p}{m}{m}  h_{ij} \\
& + \frac{c(n+1)}{2} \vddud{i}{p}{p}{j} - \frac{c(n+1)}{2}  \vdudd{p}{p}{i}{j}  + \frac{c(n+1)}{2} \vddud{j}{p}{p}{i} . 
\end{aligned}
\end{align}
\end{prop}
\begin{proof}
Since $g \longmapsto \Delta_h \beta_h (g(s))$ is a linear map, 
\begin{align*}
\partial_s ( \Delta_h \beta_h (g(s)) )|_{s=0} = \Delta_h \beta_h ( v ).
\end{align*}
If $w$ is a 1-form,  
\begin{align*}
(\delta^*_{g(s)} w)_{ij} &= \frac{1}{2} (w_{i,j} + w_{j,i}) = \frac{1}{2} (\partial_j w_{i} + \partial_i w_{j}) - \Gamma_{ij}^p w_p,
\end{align*}
so that 
\begin{align} \label{eqn:var-delstar}
\partial_s (\delta^* w)_{ij} |_{s=0} &= \delta_{g(s)}^* \left(\partial_s w_{ij} |_{s=0}\right) - (\partial_s \Gamma_{ij}^p)  w_p|_{s=0}.
\end{align}
	
Applying this to $w(s) = \Delta_h \beta_h (g(s))$, and noting that $\beta_h( g(0)) = \beta_h(h) = 0$, to get
\begin{align} \label{eqn:loc1}
\partial_s  ( \delta_{g(s)}^* ( \Delta_h \beta_h (g(s)) )|_{s=0} = \delta_h^* \Delta_h \beta_h(v).
\end{align}
	
On the other hand, we note that
\begin{equation*}
\begin{split}
[\beta_h(v)]_{i} & = - \vddu{i}{p}{p} + \frac{1}{2}  \vdud{p}{p}{i} \Longrightarrow \\
[\delta_h^* \beta_h(v)]_{ij} &=
 -\frac{1}{2} \vddud{i}{p}{p}{j} + \frac{1}{2} \vdudd{p}{p}{i}{j} -\frac{1}{2}  \vddud{j}{p}{p}{i} \Longrightarrow \\
[\Delta_h \delta_h^* \beta_h(v)]_{ij} &= -\frac{1}{2}  \vdduddu{i}{p}{p}{j}{m}{m} + \frac{1}{2} \vdudddu{p}{p}{i}{j}{m}{m}-\frac{1}{2} \vdduddu{j}{p}{p}{i}{m}{m}.
\end{split}
\end{equation*}
Then equation \eqref{eqn:loc1} together with Lemma \ref{lemma:commute-delta-div} yield equation \eqref{eqn:lin-gauge}.
\end{proof}

Using similar calculations, we obtain the following: 

\begin{prop}
\begin{align*}
\partial_s [ \delta_{g(s)}^* \delta_h g ]_{ij}|_{s=0} &= -\frac{1}{2} ( \vddud{i}{p}{p}{j} + \vddud{j}{p}{p}{i} ),\\
\partial_s [ \delta_{g(s)}^* d (\tr^h g)_{ij} ]_{s=0} &= \vdudd{p}{p}{i}{j}.
\end{align*}
\end{prop}

\ 

In the next proposition we compute the linearization of the `quadratic curvature' terms appearing in equation \eqref{eqn:def-mbach-3.5}.  To do that, we introduce the following notation:
\begin{align*}
Q_1 &:= - 2 {\R_{jki}}^m {\P_{m}}^k, \\
Q_2 &:= - {\Rc_{j}}^m \P_{im}, \\
Q_3 &:=  - \P^{kl} (\P \owedge g)_{kijl}.
\end{align*}
\begin{prop}
If $g(s)$ is a 1-parameter family of metrics with $h = g(0)$ a metric with constant sectional curvature $c$ and $v = g'(0)$, then  
 \begin{align*}
	&D[Q_1+Q_2+Q_3]_h v_{ij} \nonumber
	\\&= \frac{c}{n-2} \vdudu{p}{m}{m}{p} h_{ij} - \frac{c}{n-2} \vdudu{p}{p}{m}{m} h_{ij} + \frac{cn}{2(n-2)} (\vdddu{i}{j}{p}{p} - \vddud{i}{p}{p}{j} - \vddud{j}{p}{p}{i} + \vdudd{p}{p}{j}{i})\\
	& - \frac{c^2 n}{n-2} v_{ij} + \frac{c^2}{n-2} \vdu{p}{p} h_{ij}. \nonumber
	\end{align*}
\end{prop}
\begin{proof}
We begin by linearizing each of the quadratic terms $Q_I, I = 1,2,3$, and then writing the result in terms of the linearization of the curvature quantities appearing in Proposition
\ref{prop:var-basic-curv}. Combining these terms and inserting curvatures from that proposition leads to the result. This computation is straightforward and tedious, 
but is simplified since we evaluate all terms at the constant curvature metric $h$. 
	
	\
	
For $Q_1$, we note that
\begin{align*}
\riemdddu{j}{k}{i}{m} {\P_{m}}^k = g^{k \ell} \riemdddu{j}{k}{i}{m} \P_{m\ell};
\end{align*}
hence
\begin{align*}
\partial_s [ {\R_{jki}}^m {\P_{m}}^k ] &= -g^{k p} g^{\ell q} v_{p q} \riemdddu{j}{k}{i}{m} \P_{m\ell} + g^{k \ell} (\partial_s \riemdddu{j}{k}{i}{m}) \P_{m\ell} + g^{k \ell} 
\riemdddu{j}{k}{i}{m} (\partial_s \P_{m\ell}).
\end{align*}
So evaluating at $g(0) = h$, we get finally that
\begin{align} \label{eqn:loc-t1-1}
\begin{aligned}
\partial_s [ {\R_{jki}}^m {\P_{m}}^k ]_{s=0} &= -\frac{c^2}{2} ( v_{ij} - \tr^h v \, h_{ij} ) + \frac{c}{2} \delta^k_m (\partial_s \riemdddu{j}{k}{i}{m}) + c  \partial_s \P_{ij} - 
c h^{pq} \partial_s \P_{pq} \, h_{ij} \\ 
&= -\frac{c^2}{2} ( v_{ij} - \tr^h v \, h_{ij} ) - \frac{c}{2}  (\partial_s \Rc_{ij}) + c  \partial_s \P_{ij} - c h^{pq} \partial_s \P_{pq} \, h_{ij}.
\end{aligned}
\end{align}
Using equation \eqref{eqn:loc-t1-1}, we obtain the variation of the first term,
\begin{align} \label{eqn:loc-DQ1}
D[Q_1]_h v_{ij} = c^2 ( v_{ij} - \vdu{p}{p} \, h_{ij} ) + c  (\partial_s \Rc_{ij}) - 2c \partial_s \P_{ij} + 2c h^{pq} \partial_s \P_{pq} \, h_{ij}.
\end{align}

\
	
Next, for $Q_2$, we start from ${\Rc_{j}}^m \P_{im} = g^{pm} \Rc_{jp} \P_{im}$, and calculate the variation at $s=0$:
\begin{align}
\label{eqn:var Q2}
\partial_s [{\Rc_{j}}^m \P_{im}]|_{s=0} &= -h^{pa} h^{mb} v_{ab} \Rc_{jp} \P_{im} + h^{pm} (\partial_s \Rc_{jp}) \P_{im} + h^{pm} \Rc_{jp} (\partial_s \P_{im}).
\end{align}
Inserting $g(0) = h$, equation \eqref{eqn:var Q2} becomes
\begin{align*}
\partial_s [{\Rc_{j}}^m \P_{im}]|_{s=0} 
&= - \frac{c^2(n-1)}{2} v_{ij} + \frac{c}{2} (\partial_s \Rc_{ij}) +  c(n-1) (\partial_s \P_{ij}),
\end{align*}
so that
\begin{align}\label{eqn:loc-DQ2}
D[Q_2]_h v_{ij} = \frac{c^2(n-1)}{2} v_{ij} - \frac{c}{2} (\partial_s \Rc_{ij}) -  c(n-1) (\partial_s \P_{ij}).
\end{align}

\
	
For $Q_3$, we write
\begin{align*}
\P^{kl} (\P \owedge g)_{kijl} &= g^{kp} g^{lq} \P_{kl} (\P \owedge g)_{pijq},
\end{align*}
and then compute that
\begin{align*}\begin{aligned}
\partial_s [\P^{kl} (\P \owedge g)_{kijl}] &= -g^{ka} g^{pb} v_{ab}  g^{lq} \P_{kl} (\P \owedge g)_{pijq} - g^{kp} g^{la} g^{qb} v_{ab} \P_{kl} (\P \owedge g)_{pijq}\\
&+g^{kp} g^{lq} (\partial_s \P_{kl}) (\P \owedge g)_{pijq} + g^{kp} g^{lq} \P_{kl} \partial_s (\P \owedge g)_{pijq}.
\end{aligned}
\end{align*}
Evaluating at $h$ and using the fact that $(\P \owedge g)_{kijl}|_{s=0} = \frac{c}{2} (h \owedge h)_{kijl} = c (h_{kl} h_{ij} - h_{jk} h_{il})$, we find
\begin{align} \label{eqn:loc-T4-1}
\begin{aligned}
\partial_s [\P^{kl} (\P \owedge g)_{kijl}]_{s=0} &= -c^2 (( \tr^h v) h_{ij} - v_{ij}) \\
&+c h^{kl} (\partial_s \P_{kl})h_{ij} - c \partial_s \P_{ij} + \frac{c}{2} h^{kl} \partial_s (\P \owedge g)_{kijl}.
\end{aligned}
\end{align}
We now use $\partial_s (\P \owedge g) = (\partial_s \P) \owedge g + \P \owedge v$ to obtain the variation at $s=0$:
\begin{align}
\label{eqn:var Q3 piece}
h^{kl} \partial_s (\P \owedge g)_{kijl} |_{s=0}&= h^{kl} (\partial_s \P_{kl}) \, h_{ij}  + \frac{c}{2} \tr^h v \, h_{ij} + (n-2) \partial_s \P_{ij} + \frac{c(n-2)}{2} v_{ij}.
\end{align}
Replacing the quantity on the left-hand side of  \eqref{eqn:var Q3 piece} in the last summand in \eqref{eqn:loc-T4-1}, we have
\begin{align*} 
\begin{aligned}
\partial_s [\P^{kl} (\P \owedge g)_{kijl}]_{s=0} &= -c^2 (( \tr^h v) h_{ij} - v_{ij}) +c h^{kl} (\partial_s \P_{kl})h_{ij} - c \partial_s \P_{ij} \\
&+ \frac{c}{2} \left( h^{kl} (\partial_s \P_{kl}) \, h_{ij}  + \frac{c}{2} \tr^h v \, h_{ij} + (n-2) \partial_s \P_{ij} + \frac{c(n-2)}{2} v_{ij} \right) \\
&= \left( -\frac{3 c^2}{4} \tr^h v  + \frac{3 c}{2} h^{kl} \partial_s \P_{kl} \right) h_{ij} +\frac{c(n-4)}{2} \partial_s \P_{ij} + \frac{c^2 (n+2)}{4} v_{ij},
\end{aligned}
\end{align*}
and so finally, 
\begin{align}\label{eqn:loc-DQ3}
D[ Q_3]_{ij} = \left( \frac{3 c^2}{4} \vdu{p}{p} - \frac{3 c}{2} h^{kl} \partial_s \P_{kl} \right) h_{ij} -\frac{c(n-4)}{2} \partial_s \P_{ij} - \frac{c^2 (n+2)}{4} v_{ij}.
\end{align}

\

Adding \eqref{eqn:loc-DQ1}, \eqref{eqn:loc-DQ2} and \eqref{eqn:loc-DQ3}, and using Proposition \ref{prop:var-basic-curv}, we see that
\begin{align*}
&D[Q_1+Q_2+Q_3]_h v_{ij} \nonumber
\\&= \frac{c}{n-2} \vdudu{p}{m}{m}{p} h_{ij} - \frac{c}{n-2} \vdudu{p}{p}{m}{m} h_{ij} + \frac{cn}{2(n-2)} (\vdddu{i}{j}{p}{p} - \vddud{i}{p}{p}{j} - \vddud{j}{p}{p}{i} + \vdudd{p}{p}{j}{i})\\
& - \frac{c^2 n}{n-2} v_{ij} + \frac{c^2}{n-2} \vdu{p}{p} h_{ij}. \nonumber
\end{align*}
\end{proof}

Combining the results above, we obtain the main result of this section: 
\begin{theorem}
\label{thm:Lin}
Let $g(s)$ be a 1-parameter family of metrics with $h = g(0)$ a metric  with constant curvature $c$ and $v = g'(0)$. If $\frac{1}{n-2}L$ is the linearization of the 
right hand side of equation \eqref{eqn:def-mbach-3}, then 
\begin{align} \label{eqn:MBF-lin}
L v_{ij} &= -\frac{1}{2} \vdddudu{i}{j}{p}{p}{m}{m} + \frac{c(n+2)}{2} \vdddu{i}{j}{p}{p} - c \vdudu{p}{m}{m}{p} h_{ij} -c \vdudu{p}{p}{m}{m} h_{ij} \nonumber \\
&+ \frac{c - 2 \mu}{2} \vddud{i}{p}{p}{j} + \frac{c -2 \mu}{2} \vddud{j}{p}{p}{i} - \frac{c - 4\nu}{2} \vdudd{p}{p}{j}{i} + c^2 \vdu{p}{p} h_{ij} - c^2 n v_{ij}.
\end{align}
\end{theorem}

\

We note that for general values of $\mu$ and $\nu$, the operator $L$ is not formally self-adjoint because of the term involving $\vdudu{p}{m}{m}{p} h_{ij}$. The following corollary readily follows from Theorem \ref{thm:Lin}:
\begin{cor} 
The $L^2$ adjoint $L^*$ of the operator $L$ of equation  \eqref{eqn:MBF-lin} is 
\begin{align*}
L^*v_{ij} &= -\frac{1}{2} \vdddudu{i}{j}{p}{p}{m}{m} + \frac{c(n+2)}{2} \vdddu{i}{j}{p}{p} - c \vdudd{p}{p}{i}{j}  -c \vdudu{p}{p}{m}{m} h_{ij} \nonumber \\ 
&+ \frac{c -2 \mu}{2} \vddud{i}{p}{p}{j} + \frac{c -2\mu}{2} \vddud{j}{p}{p}{i} - \frac{c - 4\nu}{2} \vdudu{m}{p}{p}{m} h_{ij} + c^2 \vdu{p}{p} h_{ij} - c^2 n v_{ij}.
\end{align*}
\end{cor}

\

We now show that if $\mu = -c(n-1)/2$ and $\nu = -c/4$, then  $L$ is formally self-adjoint.
Recalling the splitting of the bundle of symmetric $2$-tensors from \S \ref{sec:split}, we obtain the following Proposition:  
\begin{prop} \label{prop:L-selfadj}
Both $L$ and $L^*$ preserve the splitting \eqref{eqn:split}. If $\mu = -c(n-1)/2$ and $\nu = -c/4$, then the gauge vector is
\begin{align} \label{eqn:Z-final-gauge}
Z = \frac{1}{2} \Delta_h \beta_h (g)  -\frac{c(n-1)}{2} \delta_h g - \frac{c}{4} d( \tr^h g ),
\end{align}
and the linearization of equation \eqref{eqn:def-mbach-3.5} at $h$ is given by
\begin{align} \label{eqn:MBF-lin2}
Lv_{ij} &= -\frac{1}{2} \vdddudu{i}{j}{p}{p}{m}{m} + \frac{c(n+2)}{2} \vdddu{i}{j}{p}{p} - c \vdudu{p}{m}{m}{p} h_{ij} -c \vdudu{p}{p}{m}{m} h_{ij} \nonumber \\
&+ \frac{cn}{2} \vddud{i}{p}{p}{j} + \frac{cn}{2} \vddud{j}{p}{p}{i} - c\vdudd{p}{p}{j}{i} + c^2 \vdu{p}{p} h_{ij} - c^2 n v_{ij},
\end{align}
and $L$ is formally self-adjoint.  If $\ttv$ is a TT symmetric $2$-tensor and $f$ is a smooth function, then
\begin{align}
[L\ttv\, ]_{ij} &= -\frac{1}{2} \gendddudu{\ttv}{i}{j}{p}{p}{m}{m} + \frac{c(n+2)}{2} \gendddu{\ttv}{i}{j}{p}{p} - c^2 n \ttv_{ij}, \\
[L(fh)]_{ij} &= \left( -\frac{1}{2} \fdudu{p}{p}{m}{m} - \frac{c n }{2} \fdu{p}{p} \right) h_{ij}. \label{eqn:Lfh}
\end{align}
\end{prop}
\begin{proof}
If $\ttv$ is TT, then (for any choices of $\mu$ and $\nu$)
\begin{align} \label{eqn:LonTT}
[L\ttv\, ]_{ij} = [L^*\ttv\, ]_{ij} = -\frac{1}{2} \gendddudu{\ttv}{i}{j}{p}{p}{m}{m}+ \frac{c(n+2)}{2} \gendddu{\ttv}{i}{j}{p}{p} - c^2 n \ttv_{ij}.
\end{align}
Moreover, 
\begin{align*}
\genddduu{\ttv}{i}{j}{p}{p}{j} &= \genddduu{\ttv}{i}{j}{p}{j}{p} + \riemuudu{p}{j}{i}{s} \ttv_{sj,p} +\riemuudu{p}{j}{j}{s} \ttv_{is,p} + \riemuudu{p}{j}{p}{s} \ttv_{ij,s} \nonumber \\
&= \genddudu{\ttv}{i}{j}{j}{p}{p} + ( \riemuudu{p}{j}{i}{s} \ttv_{sj} + \riemuudu{p}{j}{j}{s} \ttv_{is} )_{,p} \\
& +  \riemuudu{p}{j}{i}{s} \ttv_{sj,p} +\riemuudu{p}{j}{j}{s} \ttv_{is,p} + \riemuudu{p}{j}{p}{s} \ttv_{ij,s}.\nonumber
\end{align*}
It follows from the fact that that $\ttv$ is TT that the first term vanishes.  Since $h$ has constant curvature, then as a consequence of \eqref{eqn:constcurvature}, all terms in this
formula which involve the curvature also vanish. Applying this result to $\ttv_{ij}$ and $\ttv_{ij,p}^{\ \ \ \ p}$, we find that the divergence of $L\ttv$ in \eqref{eqn:LonTT} vanishes.  
Finally, the expression \eqref{eqn:LonTT} also shows that $L\ttv$ is traceless, so $L\ttv$ is TT.  It follows immediately that $L\ttv$ and $L^* \ttv$ 
are $L^2$ orthogonal to both the conformal class $[h]$ and to the range of the conformal Killing operator $K$.    We recall once again that
up until now, $\mu$ or $\nu$ remain un-prescribed. 

\
 
We now consider smooth multiples of the metric, $v_{ij} = f h_{ij}$. We calculate that
\begin{align*}
[L (fh)]_{ij} &= \left( -\frac{1}{2} \fdudu{p}{p}{m}{m} - \frac{c n }{2} \fdu{p}{p} \right) h_{ij} + \left ( (c-2\mu) - n \frac{c-4\nu}{2} \right) f_{,ij},
\end{align*}
and also 
\begin{align*}
[L^* (f h)]_{ij} &= \left( -\frac{1}{2} \fdudu{p}{p}{m}{m} + \frac{2 \nu - c (n-1)}{2} \fdu{p}{p} \right) h_{ij} + \left ( (c-2\mu) - c n \right) f_{,ij}.
\end{align*}
We next choose $\mu = -c(n-1)/2$ and $\nu = -c/4$ so that the coefficients of $f_{, ij}$ on each line vanish. Then $L$ and $L^*$ preserve the bundle 
of pure trace tensors, i.e., multiples of the metric, and $L(fh)=L^*(fh).$ Since elements in the range of $K$ and TT tensors are 
both trace free, the bundle of multiples of the metric remains $L^2$ orthogonal to the other summands in equation \eqref{eqn:split} also.

\
 
In conclusion, since multiples of the metric are $L^2$-orthogonal to the image of $K$, we have $(f h, L K \alpha ) = (L^*( fh ), K \alpha) = 0$.  
Since TT tensors are $L^2$-orthogonal to the image of $K$ we have $( \ttv, L K \alpha ) = (L^*\ttv, K \alpha) = 0$.  
Therefore, $L$ maps the range of $K$ to itself. Hence we have seen that for the given choices of $\mu$ and $\nu$, $L$ and $L^*$  preserve the 
splitting \eqref{eqn:split}, and $L$ is formally self-adjoint. 
\end{proof}

\section{Linear stability at constant curvature metrics} \label{sec:l2spec}
We now carry out the $L^2$ spectral analysis of the operator $L$ from Proposition \ref{prop:L-selfadj}. Here is the strategy.  By Proposition \ref{prop:L-selfadj} and the completeness of $M$, we have that the operator $L$ is self-adjoint, and hence its spectrum lies in $\mathbb R$. To show that the spectrum of $L$ is nonpositive,
it suffices to find a constant $a \ge 0$ such that $(v, Lv) \leq -a\, \|v\|^2_{L^2}$ if $v \in C^\infty_0$.  This estimate is verified using some involved but straightforward
arguments.  

In the following, we treat the cases $c = 0, c > 0$ and $c<0$ in order of increasing difficulty.

\bigskip

\subsection {\bf Case 1: Spectral estimate if $c=0$} 
\begin{theorem} \label{thm:main1a}
For $n \geq 4$, suppose that $(M^n,h)$ is a flat orientable manifold.  Then for any compactly supported smooth symmetric $2$-tensor $v$,
\[ 
(v,Lv) \leq 0. 
\] 
Moreover if $M$ is compact, then the kernel of $L$ is finite dimensional and consists of parallel symmetric $2$-tensors.
\end{theorem}
\begin{proof} 
The finite dimensionality is immediate since $M$ is compact.  Since $h$ is flat and since covariant derivatives commute, the expression for $L$ in Proposition \ref{prop:L-selfadj} becomes simply 
$Lv = (\nabla^* \nabla)^2 v$. 

If $Lv = 0$, then $\nabla^2 v = 0$. Then taking the trace, $\nabla^*\nabla  v = 0$, and thus finally, $(v, \nabla^* \nabla v) = \|\nabla v\|^2$, so that
$v$ is parallel. 
\end{proof}

\bigskip

\subsection {\bf Case 2: Spectral estimate if $c=1$} 

This case requires a finer analysis using a decomposition of the space of symmetric $2$-tensors.  Since the gauge-adjusted flow preserves volume, 
we may restrict to infinitesimal variations $v$ for which $\int_{M} \tr^h v \, \dvol_h = 0$.

\begin{theorem} \label{thm:main1b}
Suppose that $(M^n,h)$, $n \geq 4$, is a compact orientable manifold with constant sectional curvature $c=1$ (thus $M$ is either the sphere $S^n$ or 
a lens space $S^n/\Gamma$). Let $\pi_{\mathrm{im}(K)}$ and $ \pi_{\mathrm{TT}}$ be the $L^2$ orthogonal projections consisting of the closed subspaces 
$\mathrm{im}(K)$ and of the TT tensors. Then for any smooth symmetric $2$-tensor $v$, 
\[ 
(v,Lv) \leq -n \left( \| \pi_{\mathrm{im}(K)} v \|^2  + \| \pi_{\mathrm{TT}} v \|^2\right). 
\]
Consequently the kernel of $L$ is the $(n+1)$-dimensional space $\{v = f h,\ \Delta_{h} f + nf = 0\}$.
\end{theorem}
\begin{proof}
We decompose an arbitrary element  $ v \in \Sigma^2(T^*M) $ as 
\[ 
v = K \alpha + f h + \ttv, 
\]
where $\alpha$ is a $1$-form, $f$ is a function and $\ttv$ is a transverse-traceless 2-tensor. Since $L$ preserves this orthogonal splitting, we have 
\begin{align*}
(v, Lv) = (K\alpha, L K \alpha) + (f h, L( fh)) + ( \ttv, L\ttv).
\end{align*}
Each term is estimated separately below: 

\

First, we write 
\begin{align*} 
(fh,L (fh)) =  (f, -\frac{n}{2}\Delta^2 f - \frac{n^2}{2} \Delta f) = -\frac{n}{2}(f, \Delta^2 f) + \frac{n^2}{2} ||\nabla f||^2.
\end{align*}
Next, we substitute in the identities
\[
\Delta^2 f = (\nabla^*\nabla)^2 f = (\nabla^*)^2\nabla^2 f - (n-1)\Delta f, \ \ \mbox{and}\ \ \ \nabla^2 f = (\nabla^2 f)^\circ + \frac{1}{n} \Delta f\, h
\]
to determine that
\begin{align*}
(fh, L(fh)) & = -\frac{n}{2} ||\nabla^2 f||^2 + \frac{n}{2} \left( -\frac{1}{n-1} ||\nabla^2 f||^2 + \frac{1}{n-1}||\Delta f||^2\right) \nonumber \\
& = -\frac{n^2}{2(n-1)} ||\nabla^2 f||^2 + \frac{n}{2(n-1)} ||\Delta f||^2 = - \frac{n^2}{2(n-1)} ||(\nabla^2 f)^\circ||^2 \leq 0.
\end{align*}

\

By \eqref{eqn:Lfh}, $L(fh ) = 0$ implies that $(\Delta + n)(\Delta f) = 0$; hence either $\Delta f = 0$, i.e., $f$ is constant, 
or else $\Delta f + n f$ is a constant.  Since $\int f = 0$, we conclude that $(\Delta + n)f = 0$.  Thus if $M = S^n$, then
$f$ lies in the $(n+1)$-dimensional eigenspace corresponding to the first nontrivial eigenvalue on $S^n$, while 
if $M = S^n/\Gamma$, then $f$ is identified with a $\Gamma$-invariant element of this same eigenspace on $S^n$.
For many finite groups $\Gamma \subset \mathrm{SO}(n)$, there are no $\Gamma$-invariant elements in this 
eigenspace, hence no solutions to $L(fh) = 0$ on such lens spaces. 

\

Next consider a traceless symmetric tensor $v$.  Applying \eqref{eqn:MBF-lin2} we find that
\begin{eqnarray*}
(v, Lv) &= ( v, -\tfrac12 (\nabla^*\nabla)^2 v - \frac{n+2}{2}\nabla^*\nabla v - n \delta \delta^* v - nv) \nonumber \\
&= -\frac{1}{2} \| \nabla^*\nabla v \|^2 - \frac{n+2}{2} \| \nabla v \|^2 - n \| \delta v \|^2 - n \|v\|^2.
\end{eqnarray*}
We apply this in two ways. First, if $\ttv$ is TT, then $L\ttv = 0$ implies $\ttv = 0$.  If $\ttv$ is orthogonal to the kernel of $L$, then dropping negative terms implies that $(\ttv, L\ttv) \leq - n \| \ttv \|^2$.   Second, if $v = K\alpha$ and $L v = 0$, then 
$v = K \alpha = 0$, and more generally, $(K \alpha, L K \alpha) \leq - n \| K \alpha \|^2 $.

\

This proves the spectral bound in this case.
\end{proof}
 
\bigskip

\subsection{\bf Case 3: Spectral estimate if $c=-1$}
\label{sec:spec hyp}

This case is considerably more involved than the other two.  As above, we may decompose an arbitrary compactly supported symmetric 2-tensor as
$v = K \alpha + f h + \ttv$, so that 
\begin{align} \label{eqn:split-spec-est}
(v, Lv) = (K\alpha, L K \alpha) + (f h, L( fh)) + ( \ttv, L\ttv).
\end{align}
We consider each of these terms in turn.

\

We begin with $\ttv$, which we write for the duration of this estimation as $v$ for simplicity.  Proposition \ref{prop:L-selfadj} shows that
\begin{align}
\label{eqn:ttv-spec-est2}
( v, L v) &= \int_{M} (-\frac{1}{2} v^{ij} \vdddudu{i}{j}{p}{p}{m}{m} - \frac{n+2}{2} v^{ij} \vdddu{i}{j}{p}{p} - 
n v^{ij} v_{ij}) dV_h \nonumber \\  &= -\frac{1}{2} \| \Delta v \|^2  +  \frac{n+2}{2} \| \nabla v \|^2  - n \|v\|^2.
\end{align}
Using the expression \eqref{eqn:commute-bilap} from Appendix \ref{sec:techappendix}, we find that
\begin{align*}
\nonumber -\frac{1}{2} \| \Delta v \|^2  &= -\frac{1}{2} \int_{M} v^{ij} (\vdddu{i}{j}{mp}{pm} -4 n v_{ij} + 4 \vdu{p}{p} h_{ij}  + 
(n-1) \vdddu{i}{j}{p}{p}) \, \dvol_h \\ &= -\frac{1}{2} \| \nabla^2 v\|^2 + 2 n \|v\|^2 + \frac{n-1}{2} \| \nabla v\|^2.
\end{align*}
Inserting this into equation \eqref{eqn:ttv-spec-est2}, we obtain
\begin{align} \label{eqn:int-est1}
(v, Lv) &=  -\frac{1}{2} \| \nabla^2 v\|^2   +  \frac{2n+1}{2} \| \nabla v\|^2 + n \|v\|^2.
\end{align}

\

We next rewrite the norm of the Hessian term using a generalization of Koiso's Bochner formula (see \cite{Lee}).

\begin{lemma} Using the notation from above, we define 
\begin{align}
\label{def:3-tensorT}
 T_{i j k} := v_{i j, k} - v_{j k, i}. 
\end{align} 
Then
\begin{align} \label{eqn:koiso}
\frac{1}{2} \|T\|^2 = \| \nabla v \|^2 - \| \delta v \|^2 - n \| v\|^2 + \| \tr^h v\|^2.
\end{align}
\end{lemma}
\begin{proof}
We calculate that
\begin{align*}
\frac{1}{2} \|T\|^2 &= \frac{1}{2} \int_{M} (v^{ij ,k} - v^{jk,i})(v_{ij,k} - v_{j k,i}) \, \dvol_h \nonumber \\
&= \| \nabla v \|^2 + \int_{M} v^{j k} \vdddu{i}{j}{k}{i}  \, \dvol_h  + \| \nabla v \|^2 + \int_{M} v^{j k} \vdddu{i}{j}{k}{i} \, \dvol_h.
\end{align*}
	
Commuting derivatives and using  that $h$ has constant negative curvature results in
\begin{align*}
\vdddu{i}{j}{k}{i} = \vddud{i}{j}{i}{k} + \riemdudu{k}{i}{i}{s} v_{sj} + \riemdudu{k}{i}{j}{s} v_{i s} = \vddud{i}{j}{i}{k} - n v_{jk} + \vdu{p}{p} h_{j k}, 
\end{align*}
and substituting this into the preceding expression, we are led to the following identity
\begin{align*}
\frac{1}{2} \|T\|^2 &= \| \nabla v \|^2 + \int_{M} v^{j k} (\vddud{i}{j}{i}{k} - n v_{jk} + \vdu{p}{p} h_{j k}) \, \dvol_h \nonumber \\
&= \| \nabla v \|^2 - \| \delta v \|^2 - n \| v\|^2 + \| \tr^h v\|^2.
\end{align*}
\end{proof}

Inspired by this calculation, we now define
\begin{align}
\label{def:4-tensorA}
A := \nabla T. 
\end{align} 
\begin{lemma} If $v$ is a compactly supported traceless tensor $v$, then
\begin{align*}
\|A\|^2 &=  2 \| \nabla^2 v \|^2 - 2 \| \nabla \delta v\|^2 - 4 n \| \delta v\|^2 - 2(n+2) \|\nabla v\|^2  -2 (n^2 + n) \|v\|^2.
\end{align*}
\end{lemma}
\begin{proof}
We first observe that
\begin{align} \label{eqn:loc-com-int}
\|A\|^2 &= \int_{M} A^{ijkm } A_{ijkm} \, \dvol_h = \int_{M} (v^{ij,km} - v^{jk,im})(v_{ij,km} - v_{j k,im}) \, \dvol_h \nonumber \\
&= 2 \| \nabla^2 v \|^2 - 2 \int_{M} v^{ij} \vdddudu{j}{k}{i}{m}{m}{k} \, \dvol_h .
\end{align}

A lengthy but straightforward calculation which involves commuting derivatives, with details appearing in Appendix \ref{sec:techappendix}, 
yields \eqref{eqn:commut-trfree}. Inserting this into equation \eqref{eqn:loc-com-int} and integrating by parts results in the equation
\begin{align*} 
\|A\|^2 &= \int_{M} A^{ijkm } A_{ijkm} \dvol_h = \int_{M} (v^{ij,km} - v^{jk,im})(v_{ij,km} - v_{j k,im}) \dvol_h \nonumber \\
&=  2 \| \nabla^2 v \|^2 - 2 \| \nabla \delta v\|^2 - 4 n \| \delta v\|^2 - 2(n+2) \|\nabla v\|^2  -2 (n^2 + n) \|v\|^2.
\end{align*}
\end{proof}

\

We now return to equation \eqref{eqn:int-est1}. If $v$ is TT, this becomes
\begin{align} \label{eqn:spec-est-0}
(v, Lv)  &= -\frac{1}{4} \| A \|^2 + \frac{(n-1)}{2} \|\nabla v\|^2 - \frac{n(n-1)}{2} \|v\|^2,
\end{align}
which by \eqref{eqn:koiso} is equivalent to
\begin{align} \label{eqn:spec-est-AT}
(v, Lv) 
&= -\frac{1}{4} ( \|A\|^2 - (n-1) \|T\|^2) .
\end{align}

\ 

It thus remains to show that $\|A\|^2 \geq (n-1) \|T\|^2$ if $v$ is TT, which we now do.
\begin{prop} \label{prop:spec-est-TT}
Using the notation from above, we have 
\begin{align*} \|A\|^2 \geq (n-1) \|T\|^2,
\end{align*}
and hence 
\begin{align} \label{eqn:lvv}
(v, Lv) \leq - \frac{(n-2)}{4} \| T \|^2.
\end{align}
The kernel of $L$ consists of TT tensors $v$ for which $T = 0$.
\end{prop}
\begin{proof}
Using the symmetry of $v$, we obtain
\[ 
T_{ijk} =  v_{ij,k} - v_{jk,i} = v_{ij,k} - v_{kj,i}.
\]
This is skew-symmetric in $i$ and $k$; hence it may be considered as a $2$-form taking values in $T^* M$:
\[ 
T = ( T_{ijk} \; dx^i \wedge dx^k) \otimes dx^j. 
\]

\

We now consider the Hodge Laplacian acting on $T^*M$-valued $2$-forms.  We denote by 
\begin{align*}
d^\nabla: \Omega^p(T^*M) \otimes \Gamma(T^*M) \longrightarrow \Omega^{p+1}(T^*M) \otimes \Gamma(T^*M)
\end{align*}
the exterior derivative on $\Omega^*(M)$ coupled to the Levi-Civita connection on $T^* M$. We write this more explicitly only
for $p = 1, 2$.   On $2$-forms, we have
\begin{align*}
d^\nabla T &= (d^\nabla T)_{ijkl} \; dx^l \wedge dx^i \wedge dx^k \otimes dx^j, \nonumber \\
[d^\nabla T]_{ijkl} &= \frac{1}{3} ( \nabla_i T_{k j l} + \nabla_k T_{l j i} + \nabla_l T_{i j k} ) \nonumber \\ 
	&= \frac{1}{3} ( T_{kjl,i} + T_{lji,k} + T_{ijk,l} ).
\end{align*}
Likewise, on $1$-forms $\eta = \eta_{ij} dx^i \otimes dx^j$,  
\begin{align*}
d^\nabla \eta &= (d^\nabla \eta)_{ijk} \; dx^k \wedge dx^i \otimes dx^j, \nonumber \\
[d^\nabla \eta]_{ijk} &= \frac{1}{2} (\eta_{kj,i} - \eta_{ij,k}).
\end{align*} 
The associated $L^2$ inner product is 
\begin{align*}
(\eta, \zeta) = \int_{M} p! \, \eta^{k_1 \ell  k_2 \cdots k_p} \zeta_{k_1  \ell k_2 \cdots k_p} \, \dvol_h 
\end{align*}
where indices are raised using $h$.  We note that these $T^*M$-valued $p$-forms such as $\eta$ and $\zeta$, denoted e.g., $\zeta_{k_1 \ell k_2 \ldots k_p}$, are fully skew in the indices $(k_1, \ldots, k_p)$ only,
but not in the remaining index $\ell$. 
Using this inner product we may determine the adjoint differentials for these quantities. For example, if $\eta$ is a $1$-form and $\omega$ is a $2$-form,
both valued in $T^*M$, then 
\begin{align*}
(d^\nabla \eta, \omega) &:= \int 2! \, (d^\nabla \eta)_{ijk} \omega^{ijk} \, \dvol_h\\
&= \int (\eta_{kj,i} - \eta_{ij,k}) \omega^{ijk} \, \dvol_h\\
&= \int -\eta_{kj} {\omega^{ijk}}_{,i} - \eta_{ij} {\omega^{kji}}_{,k} \, \dvol_h \; \; \; \mbox{(since $\omega$ is a 2-form)} \\
&= \int 1! \; \eta_{ij} ( - 2 {\omega^{kji}}_{,k} ) \, \dvol_h\\
&= (\eta, (d^\nabla)^* \omega ),
\end{align*}
where $[(d^\nabla)^* \omega]_{ij} = - 2 {\omega_{kji,}}^k$.
	
\
	
Analogously, if $T$ is a $2$-form and $\omega$ a $3$-form, both valued in $T^*M$, then 
\begin{align*}
(d^\nabla T, \omega) &:= \int 3! \, (d^\nabla T)_{ijkl} \omega^{ijkl} \, \dvol_h \\
&= \int 2! \; ( T_{kjl,i} + T_{lji,k} + T_{ijk,l} )   \omega^{ijkl} \, \dvol_h \; \; \; \mbox{(note the cancellation of 1/3)}\\
&= \int 2! \; ( -T_{kjl}  {\omega^{ijkl}}_{,i}- T_{lji}{\omega^{kjli}}_{,k}  - T_{ijk}  {\omega^{ljik}}_{,l} ) \, \dvol_h\; \; \; \mbox{(using that $\omega$ is a 3-form)} \\
&= \int 2! \; T_{ijk} \, (-3 {\omega^{ljik}}_{,l} ) \, \dvol_h\\ 
&= ( T, (d^\nabla) ^* \omega),
\end{align*}
where $[(d^\nabla)^* \omega]_{ijk} = - 3 {\omega_{ljik,}}^l$.
	
\
	
For a $2$-form $T$, 
\begin{align*}
[(d^\nabla)^* d^\nabla T]_{ijk} = - 3 h^{ml} \nabla_m [d^\nabla T]_{ljik} = -3 [d^\nabla T]_{ljik,}^{\phantom{ljik,}l} = - {T_{ijk,l}}^l - {T_{kjl,i}}^l- {T_{lji,k}}^l,
\end{align*}  
and 
\begin{align*}
[d^\nabla (d^\nabla)^* T]_{ijk} & = \frac{1}{2} ( [(d^\nabla)^* T]_{kj,i} - [(d^\nabla)^* T]_{ij,k} ) = \frac{1}{2} ( (- 2 {T_{ljk,}}^l)_i - (- 2 {T_{lji,}}^l )_{k} )  \nonumber\\
& =  - {T^l}_{jk,li} +  {T^l}_{ji,lk}. 
\end{align*}
We can thus compute $\Delta^\nabla T := -(d^\nabla (d^\nabla)^*+  (d^\nabla)^* d^\nabla) T$, using a Bochner-Weizenb\"ock type formula for $\Delta^\nabla$ as discussed in \cite{CLN}.  

\ 

To start, we group terms as 
\begin{align*}
(\Delta^\nabla T)_{ijk} &=   {T_{ijk,l}}^l + {T_{kjl,i}}^l + {T_{lji,k}}^l + {T^l}_{jk,li} -  {T^l}_{ji,lk} \nonumber \\
&=  {T_{ijk,l}}^l + {T_{kjl,i}}^l - T_{kjl,\phantom{l}i}^{\phantom{kjl,}l} + {T_{lji,k}}^l - T_{lji,\phantom{l}k}^{\phantom{lji,}l}.
\end{align*}
Next, we commute derivatives,  and we use the fact that since $v$ is transverse-traceless, $T$ is traceless over all indices, to obtain
\begin{align*}
{T_{kjl,i}}^l - T_{kjl,\phantom{l}i}^{\phantom{kjl,}l} &= \riemdudu{i}{l}{k}{s} T_{sjl} + \riemdudu{i}{l}{j}{s} T_{ksl} + \riemdudu{i}{l}{l}{s} T_{kjs} = c(n-2) T_{ijk} - T_{kij},
\end{align*}
and
\begin{align*}
{T_{lji,k}}^l - T_{lji,\phantom{l}k}^{\phantom{lji,}l} &= \riemdudu{k}{l}{k}{s} T_{sji} + \riemdudu{k}{l}{j}{s} T_{lsi} + \riemdudu{k}{l}{i}{s} T_{ljs} = (n-2) T_{ijk} - T_{jki}.
\end{align*}
Using the fact that $T_{ijk} + T_{jki} + T_{kij} = 0$, we have 
\begin{align*}
(\Delta^\nabla T)_{ijk} &=  {T_{ijk,l}}^l + (2n-3) T_{ijk}.
\end{align*}

Consequently, 
\begin{align*}
- \|d^\nabla T\|^2 - \|(d^\nabla)^* T\|^2 = (T,\Delta^\nabla T) &= \int 2! \, ( T^{ijk} {T_{ijk,l}}^l + (2n-3) T^{ijk} T_{ijk}  )\,  \dvol_h \nonumber \\
&= - 2 \| \nabla T \|^2 + 2 (2n-3) \|T\|^2,
\end{align*}
hence
\begin{align*}
\| \nabla T \|^2 - (2n-3) \|T\|^2 = \frac{1}{2} \|d^\nabla T\|^2 + \frac{1}{2} \|(d^\nabla)^*T\|^2, 
\end{align*}
or
\begin{align*}
\| \nabla T \|^2 - (n-1) \|T\|^2  = \frac{1}{2} \|d^\nabla T\|^2 + \frac{1}{2} \|(d^\nabla) ^*T\|^2 + (n-2) \|T\|^2. 
\end{align*}
Returning to \eqref{eqn:spec-est-AT}, and using the above results we arrive at
\begin{align} \label{eqn:refined-spec-est}
(v, Lv) &= -\frac{1}{4} ( \|\nabla T\|^2 - (n-1) \|T\|^2)  \nonumber \\
&= -\frac{1}{8} ( \|d^\nabla T\|^2 + \|(d^\nabla)^*T\|^2) - \frac{(n-2)}{4} \|T\|^2.
\end{align}

\ 

Finally, if $v$ is TT and $Lv = 0$, then by \eqref{eqn:refined-spec-est}, $\| T \| = 0$.  
\end{proof}

\

We now turn to symmetric $2$-tensors of the form $v = K \alpha$.  The reader should be cautioned that from this point forward
we are no longer using the twisted Laplacian $\Delta^\nabla$, but rather the trace Laplacian $-\nabla^*\nabla$ and the Hodge
Laplacian $\Delta_H = -(d d^* + d^* d)$.  Inserting $v = K \alpha$ into equation \eqref{eqn:MBF-lin2} from  Proposition \ref{prop:L-selfadj}, 
we find that 
\begin{align}
\label{eqn:LKalpha}
L K \alpha = -\frac{1}{2} (\nabla^*\nabla )^2 K \alpha + \frac{(n+2)}{2} \nabla^*\nabla K \alpha + n K \delta K \alpha - n K \alpha.
\end{align}

\

In the next Proposition we rewrite $(K \alpha, L K\alpha)$ as a sum of nonpositive terms. In preparation, we state some useful identities.
First, we recall the standard Weitzenb\"ock identity relating the trace Laplacian and the Hodge Laplacian on $1$-forms corresponding to a hyperbolic metric 
$h$ (as usual, we alert our readers to our sign convention):
\begin{align}
\nabla^*\nabla \alpha = - \Delta_H \alpha + (n-1) \alpha. 
\label{Weiz}
\end{align}
Next, combining \eqref{Weiz} with Lemma \ref{lemma:commute-delta-delstar}, with $c = -1$, gives
\begin{align}
\nabla^* \nabla  K \alpha & =  K \nabla^*\nabla \alpha + (n+1) K \alpha,  \nonumber \\
K \nabla^* \nabla \alpha & = -K \Delta_H \alpha + (n-1) K \alpha.
\label{110}
\end{align}
Finally, on traceless symmetric $2$-tensors, the $L^2$ adjoint operator $K^*$ reduces to $\delta$, so that
\begin{align*}
K^* K \alpha = \delta ( \delta^* \alpha + \frac{1}{n} (d^* \alpha)\, h) = \delta \delta^* \alpha - \frac{1}{n} d\alpha.
\end{align*}
We readily calculate that 
\begin{align}
\delta^*\alpha = \nabla \alpha - \frac12 d\alpha \ \ \Longrightarrow \ \ \delta \delta^*\alpha = \delta \nabla \alpha - \frac12 d^* d \alpha = \nabla^*\nabla \alpha - \frac12 d^* d\alpha.
\label{inter}
\end{align}
Using \eqref{Weiz}, then $\delta \delta^* \alpha = -\frac12 \Delta_H + \frac12 d d^* + (n-1)$, and consequently we obtain
\begin{align}
K^* K \alpha = -\frac12 \Delta_H \alpha + \frac{n-2}{2n} dd^* \alpha + (n-1)\alpha.
\label{111}
\end{align}

\begin{prop} \label{prop:spec-est-imK}
For a symmetric $2$-tensor of the form $v = K \alpha$, 
\begin{align} \label{eqn:spec-est-imK}
(K\alpha, LK\alpha) & = -\frac14 ||dd^*d\alpha||^2 - \frac{n-1}{2n} ||d^* d d^*\alpha||^2 - (n-1) ||\Delta_H \alpha||^2 \nonumber \\
& \qquad - (n-1)^2 ||d\alpha||^2 - \frac{n(n-1)}{2} ||d^*\alpha||^2. 
\end{align}
In particular $ (K \alpha, L K \alpha) \leq 0$ for all $1$-forms $\alpha$.  If $L K \alpha = 0$, then $\alpha$ is a harmonic $1$-form.  Thus the kernel of $L$ on this subspace consists of $2$-tensors of the form $v = K \alpha$ for harmonic $1$-forms $\alpha$.
\end{prop}
\begin{proof}
First, it follows from equation \eqref{eqn:LKalpha} that
\begin{align}
 (K \alpha, L K\alpha) &= (K\alpha, -\frac{1}{2} (\nabla^* \nabla)^2 K\alpha + \frac{n+2}{2} \nabla^*\nabla K\alpha + n K \delta K\alpha -  n K\alpha) \nonumber \\
& = -\frac12 ||\nabla^*\nabla K\alpha||^2  + \frac{n+2}{2} (\nabla^*\nabla K\alpha, K\alpha) + n ||K^* K\alpha||^2 - n ||K\alpha||^2 \\ 
& = -\frac12 ||K \nabla^*\nabla \alpha||^2 - \frac{n}{2} ( \nabla^* \nabla \alpha, K^* K\alpha) + n ||K^* K\alpha||^2 - \frac{n-1}{2} ||K\alpha||^2,
\label{112} 
\end{align}
where the last equality relies on \eqref{110} and some algebraic simplification. 

\ 

It follows from the second part of equation \eqref{110} that 
\begin{align*}
-\frac12 \| K \nabla^*\nabla \alpha \|^2  & = -\frac12 \| K \Delta_H \alpha - (n-1) K \alpha \|^2 \\ 
&=-\frac12 (K^*K \Delta_H \alpha, \Delta_H \alpha) + (n-1)  (K^* K \alpha, \Delta_H \alpha) -\frac{(n-1)^2}{2} || K \alpha||^2,
\end{align*}
and
\begin{align*}
-\frac{n}{2}( \nabla^*\nabla \alpha, K^* K\alpha) = \frac{n}{2} ( \Delta_H \alpha, K^* K\alpha) - \frac{n(n-1)}{2} ||K\alpha||^2;
\end{align*}
hence inserting these into \eqref{112} results in
\begin{align*}
 (K \alpha, L K\alpha) =-\frac12 ( K^*K \Delta_H \alpha, \Delta_H \alpha) + \frac{3n-2}{2}( \Delta_H \alpha, K^* K\alpha) + n ||K^*K\alpha||^2 - n(n-1) ||K\alpha||^2.
\end{align*}
Writing the four terms on the right-hand side as $I + II + III + IV$, we shall expand each, using \eqref{111}, and then collect like terms in the sum at the end. We define 
\begin{align*}
I & := -\frac12 (  (-\tfrac12\Delta_H + \frac{n-2}{2n} dd^* + (n-1)) \Delta_H \alpha, \Delta_H \alpha) \\
& = \frac14 (\Delta^2_H \alpha, \Delta_H \alpha) - \frac{n-2}{4n}( dd^* \Delta_H \alpha, \Delta_H \alpha) - \frac{n-1}{2} ||\Delta_H \alpha||^2 \\
& = -\frac{1}{4} ||d d^* d \alpha||^2 - \frac{n-1}{2n} ||d^* d d^*\alpha||^2 - \frac{n-1}{2} ||\Delta_H \alpha||^2; 
\end{align*}
\begin{align*}
II & := \frac{3n-2}{2} (\Delta_H \alpha, -\tfrac12 \Delta_H \alpha + \frac{n-2}{2n} dd*\alpha + (n-1)\alpha) \\
& = -\frac{3n-2}{4}||\Delta_H \alpha||^2 - \frac{(n-2)(3n-2)}{4n} ||dd^* \alpha||^2 - \frac{(3n-2)(n-1)}{2}( ||d\alpha||^2 + ||d^* \alpha||^2);
\end{align*}
\begin{align*}
III & := n || -\tfrac12 \Delta_H \alpha + \frac{n-2}{2n}dd^*\alpha + (n-1)\alpha ||^2 \\
& =  \frac{n}{4} ||\Delta_H \alpha||^2 + \frac{(n-2)^2}{4n} ||dd^* \alpha||^2 + n(n-1)^2 ||\alpha||^2 \\
& \qquad - \frac{n-2}{2} (\Delta_H \alpha, dd^* \alpha) - n(n-1) (\Delta_H \alpha, \alpha) + (n-1)(n-2) (dd^* \alpha, \alpha) \\
& = \frac{n}{4} ||\Delta_H \alpha||^2 + \frac{ (n-2)(3n-2)}{4n} ||dd^*\alpha||^2  \\
& \qquad + n(n-1) ||d\alpha||^2 + 2(n-1)^2 ||d^*\alpha||^2 + n(n-1)^2 ||\alpha||^2;
\end{align*}
and
\begin{align*}
IV &:=  -n(n-1)( -\tfrac12 \Delta_H \alpha + \frac{n-2}{2n} dd^*\alpha + (n-1)\alpha, \alpha)  \\
& = -\frac{n(n-1)}{2} ||d\alpha||^2 - (n-1)^2 ||d^* \alpha||^2 - n(n-1)^2 ||\alpha||^2.
\end{align*}

Collecting all of these terms, we find that
\begin{align*}
(K\alpha, LK\alpha) & = -\frac14 ||dd^*d\alpha||^2 - \frac{n-1}{2n} ||d^* d d^*\alpha||^2 - (n-1) ||\Delta_H \alpha||^2 \\ 
& \qquad - (n-1)^2 ||d\alpha||^2 - \frac{n(n-1)}{2} ||d^*\alpha||^2. 
\end{align*}

This proves that $(K \alpha, L K \alpha) \leq 0$.   If $L v = 0$ with $v = K \alpha$, then each summand vanishes separately; hence 
$d\alpha = d^* \alpha = 0$, and consequently $\alpha$ is an $L^2$ harmonic $1$-form.  
\end{proof}

\

To complete the spectral analysis for the linearized gauge-adjusted Bach operator for the case in which $c=-1$, we discuss separately the two cases in which the hyperbolic 
manifold is either compact, or else is all of $\mathbb H^n$. We then show how it implies a spectral bound on Poincar\'e-Einstein spaces that are perturbations of $\mathbb H^n$.

\medskip

\noindent{\bf Spectral bounds on compact hyperbolic manifolds:}
\begin{theorem} \label{thm:main1c}
Suppose that $(M^n, h)$, $n \geq 4$, is a compact hyperbolic manifold. Using the notation as above, if $v$ is any symmetric $2$-tensor, then 
\begin{equation*}
(v, Lv) \leq 0.
\end{equation*}
Moreover, there is a spectral gap at $0$ and the $L^2$ kernel of $L$ consists of the following two types of symmetric 2-tensors:
\begin{itemize}
\item $v = K \alpha$ where $\alpha$ is a harmonic $1$-form.  The dimension of this portion of the nullspace equals $\dim H^1(M, \mathbb R)$.
\item Non-trivial tensors $v$ such that $T = 0$; i.e., $v_{ij,k} - v_{jk,i} = 0$. These are the infinitesimal conformally flat deformations of the
hyperbolic structure; cf.\ \cite{Lafontaine}.
\end{itemize}
\end{theorem}
\begin{proof}[Proof of Theorem \ref{thm:main1c}] 
It follows from the compactness of $M$ that  the spectrum of $L$ is discrete, which shows that there is a spectral gap at $0$.  If we write $v = K \alpha + f h + \ttv$, then we obtain
\begin{align*} 
(v, Lv) = (K\alpha, L K \alpha) + (f h, L(fh)) + ( \ttv, L\ttv).
\end{align*}
By Proposition \ref{prop:spec-est-imK},  $(K\alpha, LK\alpha) \leq 0$ and it vanishes precisely if $\alpha$ is harmonic.  Next, 
Proposition \ref{prop:spec-est-TT} shows that $( \ttv, L\ttv) \leq 0$, and this term vanishes precisely if $T = 0$.   Finally,
if $v = f h$, then by \eqref{eqn:Lfh}, 
\begin{align*} 
(f h, L(fh)) &=  \int_{M}  n f \left( -\frac{1}{2} \Delta^2 f + \frac{n}{2} \Delta f\right)\, \dvol_h 
= -\frac{n}{2} \left( \| \Delta f \|^2 +  \|\nabla f\|^2\right) \leq 0.
\end{align*}
This quantity vanishes if and only if $f$ is constant, but since we only consider volume-preserving variations, the kernel in this component is trivial.
\end{proof}

\medskip

\noindent{\bf Spectral bounds on hyperbolic space:}
Here we suppose that $(M, h)$ is the standard hyperbolic space $\mathbb H^n$.  Our main result is the following
\begin{theorem} \label{thm:main1nc} There is a constant $a(n) > 0$ such that if $v$ is any $\mathcal C^\infty_0$ symmetric $2$-tensor on $\mathbb H^n$ with $n \geq 4$, then 
\begin{equation}
(v, Lv) \leq -a(n) \|v\|^2.
\label{spbound}
\end{equation}
In particular, the $L^2$ kernel of $L$ is trivial. 
\end{theorem}
\begin{proof}
We write $v = fh + \ttv + K \alpha$ as above so that $(v, Lv) = (fh, L(fh)) + (\ttv, L \ttv) + (K\alpha, L K \alpha)$, and handle these three terms in turn.

\ 

First, if $v = f h$, then by \eqref{eqn:Lfh}
\begin{align} \label{eqn:hyp-est-fnc}
(f h, L(fh)) = 
-\frac{n}{2} \| \Delta f \|^2 - \frac{n}{2} \|\nabla f\|^2 \leq 0.
\end{align}
Thus $L( fh) = 0$ implies that $\nabla f = 0$, i.e., $f$ is constant, hence $f$ does not lie in $L^2$.  The inequality \eqref{spbound} for this scalar component follows from the well-known estimate
$\| \nabla f \|^2 \ge \frac{(n-1)^2}{4} \|f\|^2$ for $L^2$ functions on $\mathbb H^n$ (see for example \cite{McKean}).

\ 

Next, by Proposition \ref{prop:spec-est-TT}, $( \ttv, L\ttv) \leq 0$, and if $L \ttv = 0$ for some $L^2$ TT tensor $\ttv$, then $T = 0$. We claim that in fact $\ttv = 0$. 
Indeed, by \cite[Remark 6.2]{Delay}, the spectrum of $\nabla^*\nabla$ on TT tensors is purely absolutely continuous and occupies the entire ray $[ \frac{(n-1)^2}{4} + 2,\infty)$; hence we have
\begin{align} \label{eqn:hyp-lap-spec-est}
(\tfrac14 (n-1)^2 + 2) ||\ttv||^2 \leq \| \nabla \ttv \|^2. 
\end{align}
Now if $T = 0$ then $A = \nabla T = 0$ as well, so that \eqref{eqn:spec-est-0} becomes
\begin{align*}
(\ttv, L\ttv) \geq \left( \frac12 (n-1) ( \frac{(n-1)^2}{4} + 2) - \frac{n(n-1)}{2}\right) ||\ttv||^2 = \frac{ (n-1)(n-3)^2}{8} ||\ttv||^2.
\end{align*}
This contradicts the inequality $(\ttv, L\ttv) \leq 0$ if $n \geq 4$ unless $\ttv = 0$, hence the $L^2$ kernel of $L$ on TT tensors is trivial. 

\

Regarding the spectral bound, rewrite \eqref{eqn:hyp-lap-spec-est} as
\begin{align*}
- \|\nabla \ttv\|^2 \leq - \left( \frac{(n-1)^2}{4} + 2  \right) \|\ttv\|^2.
\end{align*}
Taking this together with \eqref{eqn:lvv} and \eqref{eqn:koiso}, we see that
\begin{align} \label{eqn:fin-est-tt}
(\ttv,L\ttv) \leq \frac{-(n-2)}{4} \|T\|^2 &= \frac{(n-2)}{2} \left( -\| \nabla \ttv\|^2 +  n \|\ttv\|^2 \right) \leq-\frac{(n-2)(n-3)^2}{8} \|\ttv\|^2.
\end{align}

\ 

Finally, Proposition \ref{prop:spec-est-imK} shows that $(K\alpha, L K \alpha) \leq 0$ and any element in the $L^2$ kernel in this component is of the form $K \alpha$ 
where $\alpha$ is an $L^2$ harmonic $1$-form.  However, it is well-known that there are no such $1$-forms on $\mathbb H^n$ if $n \geq 3$, so this component also has no $L^2$ kernel.  
We now observe, using the bound $\|d\alpha\|^2 + \|d^* \alpha\|^2 \geq \tfrac14 (n-3)^2 \|\alpha\|^2$ for $L^2$ $1$-forms on hyperbolic space,   that
\begin{align*}
\|K\alpha\|^2 & = ( -\tfrac12 \Delta_H \alpha + \frac{n-2}{2n} dd^* \alpha + (n-1)\alpha, \alpha) = \frac12 \|d\alpha\|^2 + \frac{n-1}{n}\|d^*\alpha\|^2 + (n-1)\|\alpha\|^2 \\ 
& \leq 5(n-1) ( \|d\alpha\|^2 + \|d^*\alpha\|^2).
\end{align*}
Finally, the main estimate of Proposition \ref{prop:spec-est-imK} shows that
\[
(K\alpha, L K\alpha) \leq -\frac{n(n-1)}{2} ( \|d\alpha\|^2 + \|d^*\alpha\|^2  )\leq - \frac{n}{10} \|K\alpha\|^2,
\]
which is the desired estimate if $v = K\alpha$. 

\ 

This completes the proof, taking $a(n) = \min \{ n(n-1)^2/8, \,n/10,\, (n-2)(n-3)^2/8\}.$  Note that $a(n) \geq \frac{1}{4}$ for  $n \geq 4$. 
\end{proof}

\subsection{Spectral bounds on Poincar\'e-Einstein space}
\label{sec:spec bds PE}

\

We conclude this section with the spectral bounds and linear stability for the operator on Poincar\'e-Einstein spaces which are small perturbations of hyperbolic space.  We recall that a metric $g$ on a compact manifold with boundary
$M$ is called conformally compact if it takes the form $g = \rho^{-2} \bar{g}$, where $\rho$ is a defining function for $\del M$ and $\bar{g}$ is a smooth metric up to $\del M$. The prototype is
the hyperbolic metric on the ball. The space $(M,g)$ is called Poincar\'e-Einstein if $g$ is conformally compact and Einstein, so that $\mathrm{Ric}\, g = -(n-1) g$.  Any Poincar\'e-Einstein metric
has a `boundary value' called its conformal infinity. This is the conformal class $\mathfrak c(g) = [ \rho^2 g|_{T \del M}]$ (note that since $\rho$ and $\bar{g}$ are only defined up to a nonvanishing
scalar factor, this boundary value is also defined only up to such a factor, or in other words, only its conformal class is well-defined). A  now classical result by Graham and Lee \cite{GL} 
states that any conformal class $\mathfrak c$ sufficiently close to the standard round conformal class $\mathfrak c_0$ on $S^{n-1}$ is the conformal infinity of a unique Poincar\'e-Einstein
metric $g$ which is uniformly close to the hyperbolic metric.

\begin{cor}
Let $g$ be a Poincar\'e-Einstein metric on the ball $\mathbb{B}^n$ which is sufficiently uniformly close to the hyperbolic metric $h$ as stated above, and with conformal infinity $\mathfrak c(g)$ any sufficiently
smooth conformal class near to the round conformal class on $S^{n-1}$. Then there exists a constant $a(n,\epsilon) > 0$ such that
\[
(v, Lv) \leq -a(n,\epsilon) ||v||^2
\]
for any $L^2$ symmetric $2$-tensor on $M$, where $L$ is the linearized gauge-adjusted Bach operator computed at the Poincar\'e-Einstein metric $g$. 
\label{cor1}
\end{cor}
This follows immediately since the coefficients of $L$ and the $L^2$ inner product depend smoothly on the metric $g$, so small changes in the metric in a $\mathcal C^{k,\alpha}$ norm with respect to $h$ for sufficiently large $k$ result in only small changes in this spectral bound. 

\section{Dynamical Stability of the Bach flow} \label{sec:dynamical}
In this section we discuss one way to pass from linear stability to dynamical stability for the gauge-adjusted Bach flow on $\mathbb H^n$. This
method has been discussed in detail elsewhere for Ricci flow (see \cite{GIK} and \cite{Knopf}), as well as for the ambient obstruction flow, \cite{BGIM}, 
so we restrict ourselves to sketching the main points here.

\

First, as we recalled at the end of the last section, the well-known result of Graham and Lee \cite{GL} states that there is an infinite dimensional family of Poincar\'e-Einstein metrics 
which arise as deformations of the hyperbolic metric. These are parametrized by their conformal infinities, i.e., conformal classes on $S^{n-1}$ 
near to the standard one. Each such metric is a fixed point of the Bach flow.  The existence of this family means that we should only expect a dynamical stability result for
the hyperbolic metric amongst metrics which decay asymptotically to $h$, and hence all share the same conformal infinity. We thus consider metrics $g$ such that $g-h$
lies in some weighted (little) H\"older space, as in \cite{BGIM2, BGIM}.

\

Following \cite{BGIM2}, let $M$ be the open unit ball $\mathbb{B}^n$ with linear coordinates $x = (x^1, \ldots, x^n)$, and write $\rho = \frac{1-|x|^2}{2}$.  
The hyperbolic metric thus takes the form $h = \rho^{-2} |dx|^2$.  We consider the space $\mathcal C^{k,\alpha}(M, \Sigma^2(T^*M),h)$ consisting of 
limits with respect to the H\"older norm of order $(k+\alpha)$, $\|\cdot\|_{k+\alpha}$ of smooth compactly supported symmetric $2$-tensors with respect to $h$.  We also define the weighted space
\begin{equation}
\label{wghtd norm def}
\mathcal C^{k,\alpha}_{\mu}(M, \Sigma^2(T^*M), h) = \rho^{\mu} \mathcal C^{k,\alpha}(M,\Sigma^2(T^*M), h),
\end{equation}
and observe that if $g$ is a metric such that $g - h \in \mathcal C^{k,\alpha}_{\mu}(M,h)$ for some $\mu > 0$, then the conformal infinity of $g$ is the
standard one.   For simplicity, we omit explicit mention of the tensor bundle in the notation below. 

\

Write the gauge-adjusted Bach flow \eqref{eqn:def-mbach-3.5} as $\partial g/\partial t = F(g)$. Thus $F(h) = 0$, and $DF_h = L$ is
the operator in Proposition \ref{prop:L-selfadj} (with $c=-1$).  This linearization $L$ is a fourth-order uniformly degenerate elliptic operator, and
is admissible in the sense of \cite[Definition 2.0]{BGIM2}; thus 
\[
L: \rho^{\mu} \mathcal C^{k+4,\alpha}(M,h) \longrightarrow \rho^{\mu} \mathcal C^{k,\alpha}(M,h)
\]
is bounded.  Assuming that $||g-h||_{\mathcal C^{k+4,\alpha}_\mu}$ is sufficiently small, we can expand $F$ in a Taylor series as 
\begin{align*}
F( g ) = F( h + v ) = F(h) + L v + Q_h v,
\end{align*}
where $v = g-h$ and $Q_h$ is a polynomial in $h$, $h^{-1}$ and up to three covariant derivatives of $v$, with $Q_h$ vanishing quadratically as $v \to 0$.  
It is straightforward to check that 
\begin{align*}
F: \rho^{\mu} \mathcal C^{4+k,\alpha}(M,h) \longrightarrow \rho^{\mu} \mathcal C^{k,\alpha}(M,h)
\end{align*}
is a $\mathcal C^1$ map between these two Banach spaces.   Furthermore, referring to \cite{BGIM2} again, we determine that $L$ is sectorial on $\mathcal C^{k,\alpha}_{\mu}(M,h)$
if $\mu > 0$. Since sectoriality is an open condition, $DF_g$ is sectorial for $g$ sufficiently close to $h$ as well. 

\

To apply the asymptotic stability result in Lunardi's book \cite{Lunardi}, we must bound the spectrum of $L$ acting on these weighted H\"older spaces.  In other words,
we must transfer what we know about the $L^2$ spectral bounds to, effectively, bounds for the weighted H\"older spectrum. 
We have shown in Theorem \ref{thm:main1nc} that the $L^2$ spectrum of $L$ at $h$ is a ray $(-\infty, -b]$ for some constant $b \geq a(n) > 0$. 

We now explain how, using the structure of the Green function for $L$, we can obtain a suitable bound for the $\mathcal C^{k,\alpha}_\mu$ spectrum. 
This same argument has been applied in \cite{BGIM2}. 

\renewcommand{\Re}{\mathrm{Re} \; }

\begin{prop}
\label{prop:bdresolvent}
Fix any weight $\mu \in (\frac{n-1}{2}-r(n),\frac{n-1}{2}+r(n))$, where $r(n)>0$ is a constant depending on $n$.  Then there exists a constant $a(n)$ such that the resolvent set for $L$ acting on $\mathcal C^{k,\alpha}_\mu$
is contained in a half-plane $\Re \lambda > -a(n)$. Thus 
\[ 
\mathrm{spec}_{\mathcal C^{k,\alpha}_\mu(M, h)}(L) \subset \{\lambda \in \mathbb C: \Re \lambda \leq -a(n) \}.
\]
\end{prop}

\

We note for example that the values $r(n) = \frac{n-1}{8}$ and $a(n)=\frac{(n-2)(3n-11)(5n-13)}{128}$ suffice (see Remark \ref{rem:ind-poly}).

\

\begin{proof}
If $\lambda \not\in \mathrm{spec}_{L^2} (L)$, then $\lambda - L: L^2 \to L^2$ has a bounded inverse.  Consider the Schwartz kernel of $R_{\lambda}(L) = 
(\lambda - L)^{-1}$ as a distribution on $M \times M$.  Following \cite{Ma-edge}, this Schwartz kernel is best understood as a polyhomogeneous distribution on $(\mathbb{B}^n)^2_0$, which is by definition the blowup 
of the product space $(\mathbb{B}^n)^2$ along the boundary of its diagonal. This blowup is called the $0$-double space of the ball, and is obtained by replacing each
point of the boundary of the diagonal of $(\mathbb{B}^n)^2$ with the set of inward-pointing unit normal vectors.  The entire boundary of the diagonal $\del \mathrm{diag}((\mathbb{B}^n)^2)$
is replaced by its inward-pointing spherical normal bundle.  This process creates a new front face, and the diagonal of $(\mathbb{B}^n)^2$
lifts to this new space as a submanifold whose normal bundle has a uniform product structure up to the boundary.  The resolvent $R_\lambda(L)$
is an element of the $0$-calculus of pseudodifferential operators, which means that its Schwartz kernel has a classical pseudodifferential singularity of order $-2$ along
this lifted diagonal, and has asymptotic expansions up to all boundary faces. These polyhomogeneous expansions are described in terms
of their index sets, which are the collection of exponents appearing as powers of the defining functions for these faces in the expansions. 
The mapping properties of $R_\lambda(L)$ are determined solely by these features (i.e., the singularity along the diagonal, and the orders
of vanishing or blowup of the Schwartz kernel at the various boundary faces). We refer to \cite{Ma-edge} for more careful descriptions of
all of this, and simply quote the relevant mapping properties needed here. 

\

The expansions of $R_\lambda(L)$ at the `left' and `right' faces (these are the lifts of $(\partial \mathbb{B}^n) \times \mathbb{B}^n$ and $\mathbb{B}^n \times (\partial \mathbb{B}^n)$, 
respectively) are determined by the {\it indicial roots} of $\lambda - L$. As explained carefully in Appendix A below, these indicial roots are the values $\gamma$ such that there exists
a symmetric $2$-tensor $\kappa$ which is smooth up to $\partial \mathbb{B}^n$ for which
\[
(\lambda - L) \rho^{\gamma} ( \rho^{-2}k) = \mathcal O( \rho^{\gamma+1}).
\]
The extra factor of $\rho^{-2}$ is included so that the norm of $\rho^{-2}\kappa$ with respect to $h$ is uniformly bounded (and smooth) up to $\partial B^n$.
This is an algebraic condition dependng only on the leading coefficients of the operator $\lambda - L$, and only involves the value of $\kappa$ at each 
boundary point. 
The calculation of these indicial roots is technical, and relegated to Appendix \ref{sec:appendix}.   There are sixteen separate indicial roots (these are independent of the point 
$p \in \partial \mathbb{B}^n$) which are convenient to write as eight sets of pairs, $\gamma_i^{\pm} = \gamma_i^{\pm}(\lambda)$, $i = 1, \ldots,  8$.  Each pair 
is symmetrically arranged around the midpoint $(n-1)/2$ (which is the weight $\gamma_0$ such that $\rho^{\gamma_0}$ just barely lies outside of $L^2(\dvol_h)$).   
As we explain in Appendix \ref{sec:appendix}, if $\lambda$ lies in a right half-plane $\{\lambda: \Re \lambda > -a(n)\}$ contained in the $L^2$ resolvent set (the value 
$a(n) = (n-2)(3n-11)(5n-13)/128$ will do), then the real parts of the roots satisfy the ordering
\[
\Re \gamma_8^-(\lambda) \leq \cdots  \leq  \Re \gamma_1^-(\lambda) <  \frac{n-1}{2} < \Re \gamma_1^+(\lambda) \leq \cdots \leq \Re \gamma_8^+(\lambda).
\]

\

We now define 
\[ 
\sigma := \inf \{\Re \gamma_1^+(\lambda): \Re \lambda > -a(n)\} . 
\]
In Proposition \ref{prop:ind-poly} we show that $\sigma \geq \frac{1}{4}.$   It then follows, see \cite{Ma-edge}, that the Schwartz kernel satisfies the pointwise bounds
\begin{align*}
|R(\lambda, \rho, y, \tilde{\rho}, \tilde{y})| &\leq C_{\lambda} \rho^{\frac{n-1}{2}+\sigma}, \mbox{as } \rho\to 0, \mbox{for }  \tilde{\rho} \geq c > 0, \\
|R(\lambda, \rho, y, \tilde{\rho}, \tilde{y})| &\leq C'_{\lambda} \tilde{\rho}^{\frac{n-1}{2}+\sigma}, \mbox{as } \tilde{\rho}\to 0,\mbox{for } \rho \geq c' > 0.
\end{align*}
In fact, the result is slightly sharper, inasmuch as it shows that this Schwartz kernel vanishes at this stated rate $(n-1)/2 + \sigma$ uniformly along the
left and right faces, even up to the front face. 

These bounds may then be used to show that given any $\lambda$ with $\Re \lambda > -a(n)$, there exists a range of weights $\mu$ (depending on $\lambda)$
such that 
\begin{align*}
R_{\lambda}(L): \rho^{\mu} \mathcal C^{k,\alpha}(\mathbb{B}^n, h) \to \rho^{\mu} \mathcal C^{k,\alpha}(\mathbb{B}^n, h) 
\end{align*}
is bounded.   (Of course, the range lies in the smaller space $\rho^{\mu} \mathcal C^{k+4,\alpha}$.)    Conjugating to remove the weights, this assertion is equivalent to the 
boundedness of the mapping
\begin{align*}
u(\rho, y) \longmapsto \int (\tilde{\rho}/\rho)^{\mu} R(\lambda, \rho, y, \tilde{\rho}, \tilde{y})  u(\tilde{\rho}, \tilde{y}) \, \dvol_h
\end{align*}
on $\mathcal C^{k,\alpha}(B^n, h)$, where $\dvol_h$ is comparable to $d\tilde{\rho} \ d\tilde{y}/\tilde{\rho}^{n}$.  This last boundedness statement follows, in turn, from
elliptic regularity provided we can show it if $k = \alpha = 0$.   There are two conditions that must be met: first, the condition $\mu > \frac{n-1}{2} - \sigma$ 
means that the integral converges as $\tilde{\rho} \to 0$, while the second condition $\mu < \frac{n-1}{2} + \sigma$ implies that this expression is uniformly
bounded as $\rho \to 0$. 

We refer the reader to \cite{BGIM2} for more details on this argument in an almost identical  setting. 
\end{proof}

By a similar sort of perturbation result as in Corollary \ref{cor1}, we obtain the analogous result for Poincar\'e-Einstein spaces: 
\begin{cor}
There exists an $\vep > 0$ such that if $(M,g)$ is a Poincar\'e-Einstein space, where $M = \mathbb{B}^n$, $||g-h||_{\mathcal C^{4,\alpha}} < \vep$ and $\mathfrak c(g)$ is smooth, then for any weight $\mu$ in the same interval $(\frac{n-1}{2} - r(n), \frac{n-1}{2} + r(n))$,
there exists a constant $a(n)'$ depending on $\vep$ and $\mu$ such that 
\[
\mathrm{spec}_{\mathcal C^{k,\alpha}_\mu} (M,g) \, (L) \subset \{ \lambda \in \mathbb C:  \Re \lambda \leq -a(n)' \}.
\]
\label{cor2}
\end{cor}
This follows from the observation that the resolvent kernel $L$ depends continuously (even smoothly) on the metric $g$, and as noted in Appendix \ref{sec:appendix}, $L_g$ has the same indicial roots as $L_h$.
This means that the resolvent kernel continues to exist, provided $\vep$ is small enough, and enjoys the same boundedness properties on $\mathcal C^{k,\alpha}_{\mu}$ as the resolvent for $L_h$.  Note as well that it is unimportant whether or not the H\"older norm is measured with respect to $g$ or $h$, as the resulting H\"older norms are equivalent.

\

Using these results, we can state and sketch the proof of the main goal here, which establishes the dynamic stability of the gauge-adjusted Bach flow at the hyperbolic geometry. We use the formalism of analytic semigroups as presented in
\cite{Lunardi}.  The result is Theorem \ref{thm:main-2}, which we restate here for convenience.

\begin{theorem}[Nonlinear Stability of the Gauge-Adjusted Bach Flow at Hyperbolic Space]  Let $(M,h)$ be $\mathbb H^n$ with its standard hyperbolic metric. For any $k\geq 4$ and for
$\mu \in (\frac{n-1}{2} - r(n), \frac{n-1}{2}+r(n))$, where $r(n) > 0$ is a  constant depending on $n$, there exists $\vep > 0$ such that if  
\begin{equation*}
\| 
g_0 - h 
\|_{\mathcal C^{k,\alpha}_{\mu}} < \vep,
\end{equation*}
then the gauge-adjusted Bach flow $g_0(t)$ starting at $g_0(0) = g_0$ exists for all time and converges at an exponential rate to $h$ in the $\mathcal C^{k,\alpha}_{\mu}$ norm,
\begin{equation*}
\| g_0 (t) - h \|_{\mathcal C^{k,\alpha}_{\mu}} \leq C e^{-a(n) t}  \| g_0 - h \|_{\mathcal C^{k,\alpha}_{\mu}},
\end{equation*}
where $a(n)>0$ is a constant.
\end{theorem}

\begin{proof}
Quoting \cite[Proposition 9.1.2]{Lunardi}, the long-time existence and convergence of the gauged flow to the hyperbolic metric $h$ is proved using the spectral estimate 
of Proposition \ref{prop:bdresolvent} and the main result of \cite{BGIM2}, which establishes the sectoriality of the operator $L$.  Consequently, there exists 
a $\mathcal C^{k,\alpha}_\mu$ neighborhood of $h$ for which the gauge-adjusted flow emanating from any metric in this neighborhood exists for all time and 
converges exponentially to the hyperbolic metric. 

\end{proof}

We may immediately prove Corollary \ref{cor:main-2}:

\begin{proof}

It follows easily that the ungauged Bach flow enjoys the same convergence to $h$.   For this, we recall that the gauge vector field $Z$ in \eqref{eqn:Z-final-gauge} is 
(the metric dual of) the $1$-form
\begin{align*}  
\omega &= \frac{1}{2} \Delta_L \beta_{h} (g_0)  +\frac{(n-1)}{2} \delta_{h} g_0 + \frac{1}{4} d( \tr^{h} g_0 ).
\end{align*}
It is clear that $\omega \in \rho^{\mu} \mathcal C^{k-3,\alpha}$ and thus decays at spatial infinity.  Furthermore, \eqref{eqn:exp-decay} implies that $Z$ decays exponentially in 
time on all of the ball $\mathbb{B}^n$ as $t \to \infty$.  We can then integrate this time-dependent vector field to conclude that the associated one-parameter family of diffeomorphisms 
$\phi_t$ converges to a limiting smooth diffeomorphism $\phi_{\infty}$.  This implies that $\overline{g_0}(t) = (\phi^{-1}_t)^* g_0(t)$ is the (unique) Bach flow solution 
emanating from $g_0$. This solution exists for all time and converges to $\phi^*_{\infty} h$. 
\end{proof}

\ 

Finally, for completeness we note that Theorem \ref{thm: main P-E} and Corollary \ref{cor:main P-E} follow immediately. 

\appendix

\newcommand{\Mbar}{\overline{M}}
\newcommand{\qbar}{\overline{q}}
\newcommand{\barq}{\overline{q}}
\newcommand{\ubar}{\overline{u}}
\newcommand{\Gambar}{\overline{\Gamma}}
\newcommand{\rhob}{\overline{\rho}}
\newcommand{\gendudu}[5]{{#1}^{\phantom{#2}#3\phantom{,#4}#5}_{#2\phantom{#3},#4}}
\newcommand{\genddud}[5]{{#1}_{#2 #3, \phantom{#4} #5}^{\phantom{#2#3,}#4}}

\section{Calculation of the Indicial Roots of $L$}
\label{sec:appendix}
In this appendix we compute the indicial roots of the linearization of the gauge-adjusted Bach flow at any conformally compact
asymptotically hyperbolic metric, and in particular at the hyperbolic metric $h$. This is the operator
\begin{align}
Lv = -\frac12 (\nabla^*\nabla)^2 v + \frac{n+2}{2} \nabla^* \nabla v + & \left((\delta \delta v) h  + n \delta^* \delta v\right) - nv \nonumber \\
& + ( \Delta \tr v) h + \nabla^2 \tr v + (\tr v)\, h.
\label{galin}
\end{align}

The results are collected in Proposition \ref{prop: ind rts}. Slightly more generally, we then compute the indicial roots of $\lambda - L$, where $\lambda$ is the spectral parameter in Proposition \ref{prop:ind-poly}. 

\ 

Recall that a value $\gamma$ is an indicial root of $L$ if there exists a symmetric $2$-tensor $v$ which is smooth up to $\del M$ such that
\[
L( x^\gamma v) = \calO( x^{\gamma + 1}).
\]
Here $x$ is a boundary defining function for $\del M$, so $x \geq 0$ on $M$ with $x = 0$ and $dx \neq 0$ at $\del M$. For more general $0$-operators,
this extra vanishing in the normal direction for a particular value of $\gamma$ might only occur at a particular point $p \in \del M$; in other words, the 
indicial roots might depend on $p \in \del M$.  Fortunately, in all situations considered here, this is not the case and the indicial roots are constant
along $\del M$, so the equation above stands as is. 

\

To understand this condition, expand $v$ into its Taylor series, $v \sim \sum_j v_j x^j$, and calculate formally that for a general $\gamma \in \mathbb C$, 
\[
L( x^\gamma v)  = x^{\gamma} k(x,y) \sim \sum_j k_j(y) x^{\gamma + j}
\]
for some symmetric $2$-tensors $k_j$.  We are using a coordinate system $(x, y_1, \ldots, y_{n-1})$ here, so $y_i$ parametrizes a neighborhood in $\del M$. 
That the right hand side is of the form $x^\gamma$ times a smooth function is a consequence of
the fact that $L$ is a $0$-operator. Then $\gamma$ is an indicial root if $k_0(y) = 0$, i.e., there is some leading order cancellation.  Indeed, 
by inspection we see that
\[
L( x^\gamma v) = I(L, \gamma) v_0 x^\gamma + \calO(x^{\gamma +1}),
\]
where $I(L,\gamma)$ is a polynomial of order four in $\gamma$, called the indicial family of $L$, which takes values in the space of endomorphisms
of $\Sigma^2(T^*_p M)$. Thus $\gamma$ is an indicial root if and only if there exists some $v_0$ in this finite dimensional vector space 
such that $I(L,\gamma)v_0 = 0$. This is a sort of generalized eigenvalue equation.

\ 

We can also consider this operator $L$ with respect to any conformally compact metric $g$, and it is not hard to see that $I(L,\gamma)$ depends only
on the leading order behavior of the coefficients of $L$. This means that the computation of $I(L,\gamma)$ disregards all higher order terms 
in the Taylor series of $g$ at $x=0$, and hence is the same if we were to replace the metric by $h$.  Even further, this computation also
disregards all tangential derivatives since any $\del_{y_j}$ that appears in $L$ is accompanied by an extra factor of $x$ (this is a general
feature of $0$-operators).  In summary, if we write $L$ in local coordinates as 
\[
L = \sum_{j + |\alpha| \leq 4}  a_{j \alpha}(x,y) (x\del_x)^j (x\del_y)^\alpha  \Longrightarrow I(L) = \sum_{j \leq 4} a_{j0}(0,y_0) (x\del_x)^j,
\]
where each coefficient $a_{j\alpha}$ is a matrix depending smoothly on $(x,y)$, then the indicial roots are determined only by the
leading part
\[
I(L) = \sum_{j \leq 4} a_{j 0}(0,y) (x\del_x)^j ,
\]
which is called the {\it indicial operator} of $L$.  A value $\gamma$ is an indicial root of $L$ if and only if 
\[
\sum_{j=0}^4 a_{j0}(0,y) \gamma^j v_0 = 0
\]
for some $v_0 \in \Sigma^2( T^*_{(0,y_0)}M)$.   (As noted above, in more general situations, these indicial roots depend on $y$,
but for all operators considered here they do not.)  The indicial operator, and hence the indicial roots, of $L$ at any metric $g$ are
the same as for $L$ at the metric $h$. Thus our task reduces to computing this model indicial operator. 
One additional observation simplifies the calculations below further.  

The indicial operator is functorial in the sense that
\[
I( P_1 \circ P_2) = I(P_1) I(P_2), \ \ \ I(P^*) = I(P)^*
\]
for any $0$-operators $P$, $P_1$ and $P_2$. Since $L$ is constituted by  various compositions of the covariant derivative,
it suffices in principle to compute $I(\nabla)$ on $1$-forms, symmetric $2$-tensors and $3$-tensors which are symmetric in the first two indices.
We use this observation only implicitly in the ensuing calculations.

\ 

We first note that $L$ preserves the subspaces of pure-trace and trace-free tensors.  This is apparent from the way that $L$ is written in \eqref{galin},
since if $v$ is trace-free, then the three final terms in the expression for $L$ vanish, $\nabla^* \nabla$ preserves the space of trace-free tensors,
and $(\delta \delta v) h + n\delta^* \delta v$ is the trace-free part of $n \delta^* \delta v$.  Thus we calculate the action of $L$ on these two 
subspaces separately. Below we employ a further decomposition of the space of trace-free tensors into three further $L$-invariant subspaces. 

\

\noindent{\bf Pure-trace $2$-tensors}

\

First suppose that $v = f  h$. Then a short calculation shows that
\[
L( f h) = \left( -\frac12 \Delta^2 f + \frac{n}{2} \Delta f \right) \, h.
\]
The indicial operator of the scalar Laplacian is well-known to equal
\begin{align*}
I(\Delta) = (x\del_x)^2 + (1-n)x\del_x.
\end{align*}
Therefore, 
\begin{align*}
I(L) ( f h)  = \left(-\frac12 ( (x\del_x)^2 + (1-n)x\del_x)^2 + \frac{n}{2} ( (x\del_x)^2 + (1-n)x\del_x)\right) f\, h.
\end{align*}
For simplicity, write $I(L_{\tr})$ for the operator in parentheses acting on scalar functions, so 
\begin{align} \label{eqn:locref1}
I(L) (x^\gamma h) = \left(-\frac12 ( \gamma^2 + (1-n)\gamma)^2 + \frac{n}{2} (\gamma^2 + (1-n)\gamma)\right) x^\gamma \, h.
\end{align}
There is a final adjustment. The definition of indicial roots requires that we compute $I(L)$ acting on $x^\gamma q$ where
$q$ is smooth up to the boundary. Since $h = x^{-2} h_E$ where $h_E = dx^2 + dy^2$ is the Euclidean metric, the left
hand side of equation \eqref{eqn:locref1} is $I(L) (x^{\gamma-2} h_E)$, so the roots $\tilde{\gamma}_j$ of this polynomial correspond to indicial roots
$\gamma_j = \tilde{\gamma}_j - 2$. We have now proved that the indicial roots of $L$ acting on the space of pure-trace tensors are the roots
of the fourth order polynomial 
\begin{equation}
I(L_{\tr},\gamma ) = \left(-\frac12 ( (\gamma + 2)^2 + (1-n)(\gamma + 2))^2 + \frac{n}{2} ((\gamma + 2)^2 + (1-n)(\gamma + 2))\right),
\label{indroots-pt}
\end{equation}
which are $\gamma = -3, -2, n-3, n-2.$

\

\noindent{\bf Trace-Free $2$-tensors}

\

We now assume that $v = \ttv$ is trace-free. To simplify notation, we write $\ttv = q$; this is a symmetric $2$-tensor on $\RR^n_+$ which is smooth for $x \geq 0$.  
We work in the standard rectangular coordinates $(x,y) \in \RR^n_+$ with $x \geq 0$, and we let $0$ be the index corresponding to the $x$ direction, $\alpha, \beta, \ldots$
the indices for the tangential $y$ directions, and $i, j, \ldots$ for all directions.  We also use the summation convention.

In these coordinates, the Christoffel symbols for the hyperbolic metric $h$ are given by 
\[
\Gamma^k_{j\ell} = -x^{-1} ( \delta_j^k \delta_{\ell 0} + \delta_\ell^k \delta_{j0} - \delta^{k0}\delta_{j\ell}),
\]
which reduce to
\[
\Gamma^j_{i 0} = \Gamma^j_{0i} = -\frac{1}{x} \delta^j_i,\ \ \Gamma^0_{\alpha \beta} = \frac{1}{x} \delta_{\alpha \beta},
\]
and all other $\Gamma^k_{ij} = 0$. The $1/x$ singularity is expected since $x\nabla$ is a $0$-operator, or phrased differently, the $h$-covariant 
derivatives are of the form
\[
\nabla_0 = \del_x + \frac{1}{x} A_0, \quad \nabla_\alpha = \del_\alpha + \frac{1}{x} A_\alpha,
\]
where $A_0$ and $A_\alpha$ are matrices depending on the type of tensor on which $\nabla$ is acting.  Note that we write
$\del_0$ and $\del_\alpha$ instead of $\del_x$ and $\del_{y_\alpha}$. 
As observed above, the indicial operators do not depend on the tangential derivatives $x\del_\alpha$, so we drop these systematically,
which simplifies computations. We write $L \equiv L'$ if $L$ and $L'$ differ by higher order terms.

\

First observe that if $E = E_{i_1 \ldots i_k}$ is a tensor of order $k$, then 
\[
\nabla_0 E_{i_1 \ldots i_k}  := E_{i_1 \ldots i_k, 0} = \del_x E_{i_1 \ldots i_k} - \sum_{j=1}^k \sum_{s=0}^{n-1}  \Gamma^s_{0 \, i_j} E_{i_1 \ldots i_{j-1} s i_{j+1} \ldots i_k} 
= (\del_x + \frac{k}{x}) E_{i_1 \ldots i_k}. 
\]
The expressions for the tangential derivatives are slightly more complicated and depend on the degree of the tensor.  Thus, if $\omega$ is a $1$-form, then
\[
\nabla_\alpha \omega_0 := \omega_{0,\alpha} \equiv -\Gamma^s_{\alpha 0} \omega_s = \frac{1}{x}\omega_\alpha,\ \ \nabla_\beta \omega_\alpha := \omega_{\alpha, \beta} 
\equiv -\Gamma^s_{\alpha \beta}\omega_s = -\frac{1}{x} \delta_{\alpha \beta} \omega_0.
\]
Next, if $q_{ij}$ is a symmetric $2$-tensor, then 
\begin{align*}
\nabla_\alpha q_{00} & := q_{00,\alpha} \equiv \frac{2}{x} q_{0\alpha},\quad \nabla_\beta q_{0 \alpha} := q_{0 \alpha, \beta} \equiv \frac{1}{x} q_{\alpha \beta} - 
\frac{1}{x} \delta_{\alpha \beta}  q_{0 0}\\
& \mbox{and}\ \ \nabla_\gamma q_{\alpha \beta} := q_{\alpha \beta, \gamma} \equiv -\frac{1}{x} \delta_{\alpha \gamma} q_{0 \beta} - \frac{1}{x} \delta_{\beta \gamma} q_{0 \alpha}.
\end{align*}

In computing the indicial operator of $\nabla^* \nabla$, it suffices to only compute the `pure' second covariant derivatives:
$q_{ij,00}$, $q_{00, \alpha \alpha}$, $q_{0\alpha, \beta \beta}$ and $q_{\alpha \beta, \gamma \gamma}$.  From the observations about $\nabla_0$ above, we first see
that for any $i, j = 0, \ldots, n-1$, 
\[
q_{ij, 00} = ( \del_x + \frac{3}{x}) (\del_x + \frac{2}{x}) q_{ij},
\]
simply because $q_{ij}$ is a $2$-tensor and $q_{ij,0}$ is a $3$-tensor.    For use below, we record that
\[
q_{ij,0}^{\ \ \  0} = x^2 (\del_x + \frac{3}{x})(\del_x + \frac{2}{x}) q_{ij}= (x\del_x + 2)^2 q_{ij}.
\]
Next we compute that
\begin{align*}
q_{00, \alpha \alpha} & \equiv - \frac{1}{x} (\del_x + \frac{4}{x}) q_{00} + \frac{2}{x^2} q_{\alpha \alpha}  \\
q_{0\alpha, \beta \beta} & \equiv -\frac{1}{x}( \del_x + \frac{3}{x})q_{0\alpha} - \frac{3}{x^2}\delta_{\alpha \beta} q_{0 \beta} \\
q_{\alpha \beta, \gamma \gamma} & \equiv -\frac{1}{x} (\del_x + \frac{2}{x}) q_{\alpha \beta} + \frac{1}{x^2} ( 2 \delta_{\alpha \gamma} \delta_{\beta \gamma} q_{00}
- \delta_{\alpha \gamma} q_{\beta \gamma} - \delta_{\beta \gamma} q_{\alpha \gamma} ).
\end{align*}

Taken together, using $(\nabla^* \nabla q)_{ij} = - x^2 ( q_{ij,00} + \sum q_{ij, \alpha \alpha})$, and introducing the notation
\[
\hat{q} = \sum_\alpha q_{\alpha \alpha},
\]
we find that 

\begin{align*}
(I(\nabla^* \nabla) q)_{00} & = (P + (4n-8)) q_{00} - 2 \hat{q}  \\
(I(\nabla^*\nabla)q)_{0\alpha} & = (P + (3n-4)) q_{0\alpha} \\
(I(\nabla^* \nabla) q)_{\alpha \beta} & = (P + (2n-4)) q_{\alpha \beta} - 2 \delta_{\alpha\beta} q_{00}.
\end{align*}
where $P := - (x\del_x)^2 + (n-5) x\del_x$.

\

The other main component of $L$ is the trace-free part of $n \delta^* \delta q$.  For this calculation, set $\omega = \delta q$ and for convenience,
write $\hat{q} = \sum_\alpha q_{\alpha \alpha}$. Then using the computations above, we obtain that
\begin{align*}
\omega_0  = - x (x\del_x + (3-n)) q_{00} - x \hat{q}, \ \ \omega_\alpha \equiv - x (x\del_x + (2-n)) q_{0 \alpha},
\end{align*}
and then
\begin{align*}
(\delta^* \omega)_{00} = (\del_x + \frac{1}{x}) \omega_0, \ \ (\delta^* \omega)_{0\alpha} = \frac12 (\del_x + \frac{2}{x}) \omega_\alpha \,  \ \ 
(\delta^* \omega)_{\alpha \beta} \equiv  -\frac{1}{x} \delta_{\alpha \beta} \omega_0.
\end{align*}
These lead to the formul\ae
\begin{align*}
(\delta^* \delta q)_{00} & = - (x\del_x + 2)(x\del_x + (3-n)) q_{00} - (x\del_x + 2) \hat{q} , \\
( \delta^* \delta q)_{0 \alpha} & = - \frac12 (x\del_x + 3) (x\del_x + (2-n)) q_{0\alpha}, \\
(\delta^* \delta q)_{\alpha \beta} & =  \delta_{\alpha \beta} \left( ( x\del_x + (3-n)) q_{00} + \hat{q}\right),
\end{align*}
and
\[
(\delta \delta q) h = \left( (x\del_x + 3-n)^2 q_{00} + (x\del_x + 3-n) \hat{q}\right) (dx^2 + h_E).
\]

\

Writing $B := (\delta \delta q) h + n \delta^* \delta q$, we compute finally that
\begin{align*}
B_{00} & = (1-n) (x\del_x  + 3) \big( (x\del_x + 3-n) q_{00} + \hat{q}\big) \\
B_{0\alpha} & = -\frac{n}{2} (x\del_x + 3)(x\del_x + 2-n) q_{0\alpha} \\ 
B_{\alpha \beta} & = (x\del_x + 3)\big( (x\del_x + 3-n) q_{00}  + \hat{q}\big) \delta_{\alpha \beta}.
 \end{align*}
Note that $B$ is trace-free. 

\ 

We make one final reduction, using the fact, advertised above in this section, that the space of trace-free tensors decomposes further 
into three invariant subspaces $V_1 \oplus V_2 \oplus V_3$, where
\begin{align*}
V_1 & := \{ f ( - (n-1)dx^2 + h_E) \}, \quad V_2 := \{ q: q_{00} = q_{\alpha \beta} = 0\ \forall\, \alpha, \beta\}, \\
& \mbox{and} \quad V_3:= \{ q: q_{00} = q_{0\alpha} = 0,\ \sum q_{\alpha \alpha} = 0\}.
\end{align*}
We now write out the indicial operator $I(L)$ restricted to each one of these subspaces: 

\

We begin with $V_1.$ Observe that if $q_{00} = -(n-1)f$, $q_{\alpha \alpha} = f$
and all other $q_{ij} = 0$, then
\[
I(\nabla^* \nabla q) = \begin{bmatrix} - (n-1) Qf   & 0 \\ 0  & Q f \end{bmatrix},
\]
where $Q := - (x\del_x)^2 + (n-5) x\del_x + 4n-6$. In addition, the trace-free part of $\delta^* \delta q$ has 
\[
B_{00} = -(n-1)B_{\alpha \alpha}, \qquad B_{\alpha \alpha} = - (n-1)( (x\del_x)^2 +(5-n)x\del_x + 6-3n) f,
\]
with all of the other $B_{ij} = 0$.   Therefore, for $q \in V_1,$ $I(L)q = k$, where
\begin{align}
\nonumber k_{00} & = - (n-1) \left( -\frac12 Q^2 f + \frac{n+2}{2} Q f  +B_{\alpha \alpha} -n f\right), \\
k_{\alpha \alpha} &=  -\frac12 Q^2 f +  \frac{n+2}{2} Q f + B_{\alpha \alpha} - nf,
\label{eqn:ind rts V1}
\end{align}
and all other $k_{ij} = 0$. The indicial roots are the roots of the associated polynomial, which are $\gamma = -3, -2, n-3, n-2.$

\ 

Next we move to the space $V_2$, and so we suppose that $q$ is such that only the $q_{0\alpha}$ are nonzero. Then $I(L)q$ is calculated similarly, with
\begin{equation}
I(L)q_{0\alpha} = \left( - \frac12( P + 3n-4)^2 + \frac{n+2}{2} (P + 3n-4) - \frac{n}{2} (x\del_x + 3)(x\del_x + 2-n) - n\right) q_{0\alpha}.
\label{eqn:ind rts V2}
\end{equation}
The roots are  $\gamma = -3, -1, n-4, n-2.$

\ 

Finally, we consider $q \in V_3$, and so we suppose that $q_{00} = q_{0\alpha} = \hat{q} = 0$. Then all components of $B$ vanish and
\begin{equation}
I(L)q_{\alpha \beta} = \left(- \frac12(P + 2n-4)^2 + \frac{n+2}{2} (P + 2n-4) -n \right) q_{\alpha\beta}.
\label{eqn:ind rts V3}
\end{equation}
The roots in this case are $\gamma = -2, -1, n-4, n-3.$

\

We gather all of these results into the following proposition: 
\begin{prop}
\label{prop: ind rts} Let $(M^n,g)$ be a sufficiently smooth asymptotically hyperbolic metric.  The fourth-order formally-self-adjoint operator $L$ given in \eqref{galin}, acting on symmetric $2$-tensors smooth up to the boundary,  has sixteen real indicial roots that occur in pairs symmetric about $-(n-5)/2$ if $dx^i$ is used as a basis of $T^*M$.  These roots are the following:
\begin{align*}
\mu_{V_0}&=-3,-2,n-3,n-2,\\
\mu_{V_1}&=-3,-2,n-3,n-2,\\
\mu_{V_2}&=-3,-1,n-4,n-2,\\
\mu_{V_3}&=-2, -1, n-4, n-3.
\end{align*}
For $n \geq 4$, we thus find that the indicial radius is positive and is equal to $\frac{n-3}{2}$.
\end{prop} 
In order to apply this proposition to the gauge-adjusted Bach flow, we need two modifications.  First, we need the addition of a complex spectral parameter $\lambda$.  Second, in order to apply the results of the fourth author in \cite{Ma-edge}, we increase indicial roots by $2$ to account for the fact that $\frac{dx^i}{\rho}$ is used as a basis of $T^*M$.  The main result of this appendix is the following:

\begin{prop} \label{prop:ind-poly} Let $(M^n,g)$ be a sufficiently smooth asymptotically hyperbolic metric, let $L$ be given by \eqref{galin}, and  let $\lambda \in \mathbb{C}$.  The fourth-order formally-self-adjoint operator $\lambda - L$ acting on symmetric $2$-tensors that are smooth up to the boundary has sixteen indicial roots that occur in pairs symmetric about $-(n-1)/2$, if $\frac{dx^i}{\rho}$ is used as a basis of $T^*M$.  These roots are
\begin{align*}
\mu_{V_0} &= \frac{n-1}{2} \pm \frac{1}{2}\sqrt{ 1 + n^2 \pm 2 \sqrt{n^2 - 8 \lambda}}, \\
\mu_{V_1} &= \frac{n-1}{2} \pm \frac{1}{2}\sqrt{ n^2 +1  \pm 2 \sqrt{n^2 - 8 \lambda}},\\
\mu_{V_2} &= \frac{n-1}{2} \pm \frac{1}{2}\sqrt{ n^2 -2n + 5  \pm 4 \sqrt{n^2 -2n +1 -2 \lambda}},\\
\mu_{V_3} &= \frac{n-1}{2} \pm \frac{1}{2}\sqrt{ n^2 - 4n + 5 \pm 2 \sqrt{n^2 - 4n + 4 - 8 \lambda}}.
\end{align*}
Thus, there exists a constant $a(n) > 0$ such that if $\Re(\lambda) > - a(n)$, then the indicial radius of $\lambda - L$ is strictly positive. 
\end{prop}

\begin{proof}
Using \eqref{indroots-pt}, the indicial polynomial of $\lambda - L$ on $V_0$ in this basis convention is
\begin{align*}
I( \lambda - L ) = \lambda+\frac{\mu}{2}(\mu+1)(\mu-n)(\mu-(n-1)),
\end{align*}
and a computer algebra package shows that the roots are given by
\begin{align} \label{eqn:mu-root-0}
\mu_{V_0} = \frac{n-1}{2} \pm \frac{1}{2}\sqrt{ 1 + n^2 \pm 2 \sqrt{n^2 - 8 \lambda}}.
\end{align}

Similarly, using equations \eqref{eqn:ind rts V1}, \eqref{eqn:ind rts V2} and \eqref{eqn:ind rts V3} and solving the resulting indicial polynomial for $\lambda - L$ yields the indicial roots 
\begin{align} \label{eqn:mu-root-i}
\mu_{V_1} &= \frac{n-1}{2} \pm \frac{1}{2}\sqrt{ n^2 - 4n + 5 \pm 2 \sqrt{n^2 - 4n + 4 - 8 \lambda}}, \nonumber \\
\mu_{V_2} &= \frac{n-1}{2} \pm \frac{1}{2}\sqrt{ n^2 +1  \pm 2 \sqrt{n^2 - 8 \lambda}},  \\
\mu_{V_3} &= \frac{n-1}{2} \pm \frac{1}{2}\sqrt{ n^2 -2n + 5  \pm 4 \sqrt{n^2 -2n +1 -2 \lambda}} \nonumber.
\end{align}

We now argue that the indicial radius of $\lambda - L$ is positive for $\lambda$ in a right half-plane containing the imaginary axis.  We first write $\lambda = -\vep + i y$ for $\vep > 0$ and $y \in \bR$.  Next we consider the three distinct ``outermost'' square root expressions in equations \eqref{eqn:mu-root-0} and \eqref{eqn:mu-root-i}.  We must check that each of these expressions has non-zero real-part.  We verify this in two steps, by first checking that the expression has a non-zero real part for each fixed $y \in \bR$, and then checking that this remains true asymptotically as $y$ grows without bound.

\

For example, working with the indicial polynomial for $V_0$ (which is also that of $V_2$), if we set \[x :=\sqrt{ n^2+1 \pm 2\sqrt{n^2 +8 \vep - 8 y i}},\] then we find that 
\begin{align*}
\frac{1}{4} (x^2 - n^2 - 1)^2 - n^2 - 8 \vep = -8 y i. 
\end{align*}
This equation shows that $x$ has a zero real part only if $y=0$.  If $y=0$, then $x$ has a zero real part if 
\begin{align*}
n^2 + 1 - 2\sqrt{n^2 + 8 \vep} < 0,
\end{align*}
which is impossible if $\vep$ is sufficiently small.  A similar argument works for the other cases.  We remark that the threshold values of $\vep$ that force equality are 
\begin{align*}
\vep_{V_0} = \vep_{V_2} &= \frac{(n^2 - 1)^2}{32}, \\
\vep_{V_1} &= \frac{((n-2)^2 - 1)^2}{32}, \\
\vep_{V_3} &= \frac{((n-1)^2 - 4)^2}{32}.
\end{align*}
For $n\geq 3$, it is easy to see $\vep_{V_1} \leq \vep_{V_3} < \vep_{V_0}=\vep_{V_2}$.

\

We now consider the asymptotics of these calculations as $|y| \to \infty$.  In order for $x$  to have a zero real part asymptotically, the argument of $x$ must tend to $\pm\frac{\pi}{2}$. To see that this is impossible, note that the argument of the complex number $n^2 + 8\vep - 8 y i$ tends to $\pm \frac{\pi}{2}$ as $|y| \to \infty$.  Thus the argument of its square root tends to $\pm \frac{\pi}{4}$.  Consequently, with either choice of the sign, the argument of $x$ tends to one of $\pm \frac{\pi}{4}$ or $\pm \frac{3\pi}{4}$, concluding the argument.

\

These considerations show that the constant $\vep_{V_1}$ controls the indicial radius, since for any $\vep \in (0,\vep_{V_1})$, the real parts of the ``outermost'' square root expressions in equations \eqref{eqn:mu-root-0} and \eqref{eqn:mu-root-i} are non-zero. Thus, there exists a constant $a(n) > 0$ such that for Re($\lambda) > -a(n)$ the indicial radius of $\lambda - L$ is strictly positive. 

\end{proof}

\begin{remark}
\label{rem:ind-poly}
For illustration we set an explicit indicial radius of $\frac{n-1}{8}$ and find the corresponding $a(n)$. Let $n \ge 4$.
Using the inequality $\sqrt{1 + x} \leq 1+ \frac{x}{2}$ for $x\geq -1$, we have
\begin{align*}
2 \sqrt{ (n-2)^2 + 8 \vep } \leq 2(n-2) + \frac{8\vep}{n-2},
\end{align*}
so that
\begin{align*}
(n-2)^2+1-2\sqrt{ (n-2)^2 + 8 \vep} \geq (n-3)^2 - \frac{8 \vep}{n-2}.
\end{align*}
To guarantee an indicial radius of $\frac{n-1}{8}$ we require that
\begin{align*}
(n-2)^2+1-2\sqrt{ (n-2)^2 + 8 \vep} \geq \frac{(n-1)^2}{16},
\end{align*}
and so it would suffice for $\vep$ to satisfy 
$(n-3)^2 - \frac{8 \vep}{n-2} > \frac{(n-1)^2}{16}$. 
Thus, for $a(n):=\vep$ with
\begin{align*}
0 \leq a(n) \leq \frac{(n-2)(3n-11)(5n-13)}{128},
\end{align*}
the indicial radius of $\lambda - L$ is at least $\frac{n-1}{8}.$ 

\end{remark}
\
\section{Commutation formul\ae} \label{sec:techappendix}
We record the details of a few of the commutation formul\ae\ used in the text.

\begin{lemma} \label{lemma:commute-delta-delstar} \label{lemma:commute-delta-div}
If $(M^n,h)$ has constant sectional curvature $c$ and $\alpha$ is a  $1$-form, then
\begin{equation}\label{commutator1}
[ \delta^*, \nabla^* \nabla] \alpha = c(n+1) \delta^* \alpha + 2c (\delta \alpha)\,  h
\end{equation}
and
\begin{equation}\label{commutator2}
[K, \nabla^* \nabla] \alpha = c(n+1)  K\alpha. 
\end{equation}
All covariant derivatives and traces are calculated with respect to $h$.
\end{lemma}
\begin{proof} 
We begin by computing $[\nabla, \nabla^* \nabla] \alpha$ in coordinates:
\begin{align*}
\begin{aligned}
{\alpha_{i,jk}}^k  &= {\alpha_{i,kj}}^k + \riemdudu{j}{k}{i}{s} \alpha_{s,k} \\
&= \alpha_{i,k\phantom{k}j}^{\phantom{i,k}k}  + 2 \riemdudu{j}{k}{i}{s} \alpha_{s,k} + c(n-1) \alpha_{i,j} \\
&= \alpha_{i,k\phantom{k}j}^{\phantom{i,k}k}  + 2 c ( \delta_j^s \delta_i^k - h_{ji} h^{ks} ) \alpha_{s,k} + c(n-1) \alpha_{i,j} \\
&= \alpha_{i,k\phantom{k}j}^{\phantom{i,k}k}  + 2 c  \alpha_{j,i} - 2c {\alpha_{k,}}^k h_{ij} + c(n-1) \alpha_{i,j} ;
\end{aligned}
\end{align*}
we now symmetrize in $i, j$, to get
\begin{align*}
\begin{aligned}
\frac{1}{2} ( {\alpha_{i,jk}}^k + {\alpha_{j,ik}}^k ) 
= \frac{1}{2} ( \alpha_{i,k\phantom{k}j}^{\phantom{i,k}k} + \alpha_{j,k\phantom{k}i}^{\phantom{j,k}k})   - 2c {\alpha_{k,}}^k h_{ij} + 
c(n+1) \frac12 (\alpha_{i,j}  + \alpha_{j,i} )
\end{aligned}
\end{align*}
which is \eqref{commutator1}. Next, taking traces, we see that
\[
[ \delta, \nabla^* \nabla]\alpha = -c (n-1) \delta \alpha;
\]
hence,
\[
[K, \nabla^* \nabla]\alpha = [ \delta^* + \frac{1}{n}(\delta \, \cdot \, ) \, h, \nabla^* \nabla] = c(n+1) \delta^* \alpha + 
c (2 - \frac{n-1}{n}) (\delta \alpha)\, h = c(n+1) K\alpha
\]
as claimed.
\end{proof}

\begin{lemma} With the same assumptions as Lemma \ref{lemma:commute-delta-delstar}, if $v$ is a symmetric $2$-tensor, then
\begin{align} \label{eqn:commute-bilap}
(\nabla^* \nabla)^2 v = (\nabla^*)^2\nabla^2 v + c(n-1) \nabla^* \nabla v  + 4c^2 (\tr v)\, h - 4c^2 n v.
\end{align}
\end{lemma}
\begin{proof}
Once again calculating in coordinates, 
\begin{align*}
\vdddudu{i}{j}{p}{p}{m}{m} &= \vdddu{i}{j}{pm}{pm} + ( \riemuddu{p}{m}{i}{s} v_{sj,p} + \riemuddu{p}{m}{j}{s} v_{is,p} + \riemuddu{p}{m}{p}{s} v_{ij,s} )^{,m} \nonumber \\
&= \vdddu{i}{j}{pm}{pm} + c \vddud{p}{j}{p}{i} - c \vdddu{p}{j}{i}{p} + c \vddud{p}{i}{p}{j} - c \vdddu{p}{i}{j}{p} - c(n-1) \vdddu{i}{j}{p}{p}. 
\end{align*}
Substituting the expression
\begin{align*}
c \vddud{p}{j}{p}{i} - c \vdddu{p}{j}{i}{p} &= \riemuddu{p}{i}{p}{s} v_{sj} + \riemuddu{p}{i}{j}{s} v_{ps} = -c n v_{ij} + c \vdu{p}{p} h_{ij}
\end{align*}
into this gives 
\begin{align} \label{eqn:loc-b-1}
\vdddudu{i}{j}{p}{p}{m}{m} &=\vdddu{i}{j}{pm}{pm} -2 c^2 n v_{ij} + 2 c^2 \vdu{p}{p} h_{ij} - c(n-1) \vdddu{i}{j}{p}{p}. 
\end{align}
Commuting derivatives once more,
\begin{align} \label{eqn:loc-b-2}
\vdddu{i}{j}{pm}{pm} &= \vdddu{i}{j}{mp}{pm} + ( \riemdddu{p}{m}{i}{s} \vdd{s}{j} + \riemdddu{p}{m}{j}{s} \vdd{i}{s} )^{,pm} \nonumber \\
&= \vdddu{i}{j}{mp}{pm} + c \vddud{p}{j}{p}{i} - c \vdddu{p}{j}{i}{p} + c \vddud{p}{i}{p}{j} - c \vdddu{p}{i}{j}{p} \nonumber \\
&= \vdddu{i}{j}{mp}{pm} -2 c^2 n v_{ij} + 2 c^2 \vdu{p}{p} h_{ij},
\end{align}
so finally, combining \eqref{eqn:loc-b-1} and \eqref{eqn:loc-b-2} we obtain at \eqref{eqn:commute-bilap}.
\end{proof}

\begin{lemma} If $v$ is a traceless symmetric $2$-tensor, then
\begin{align}
\label{eqn:commut-trfree}
\vdddudu{j}{k}{i}{m}{m}{k} = \, & \vddudud{j}{k}{k}{m}{m}{i} +  2c \vddud{i}{m}{m}{j} - 2c \vdudu{p}{m}{m}{p} h_{ij} \nonumber \\ & 
+ c (n+2) \vdddu{i}{j}{m}{m} + 2c(n-1) \vddud{j}{m}{m}{i} + c^2 (n^2+n) v_{ij}.
\end{align}
 \end{lemma}
 \begin{proof} First, 
\begin{align} \label{eqn:loc-com-0}
\vdddudu{j}{k}{i}{m}{m}{k} &= \vddduud{j}{k}{i}{m}{k}{m} + \riemdddu{m}{k}{j}{s} \vdudu{s}{k}{i}{m} + \riemdudu{m}{k}{k}{s} \vdddu{j}{s}{i}{m} + \riemdddu{m}{k}{i}{s} \vdudu{j}{k}{s}{m} + \riemddud{m}{k}{m}{s} \vdudu{j}{k}{i}{s}\nonumber \\
&= \vddduud{j}{k}{i}{m}{k}{m} +\riemdddu{m}{k}{j}{s} \vdudu{s}{k}{i}{m} + c(n-1) \vdddu{j}{s}{i}{s}  + \riemdddu{m}{k}{i}{s} \vdudu{j}{k}{s}{m} -c(n-1) \vdddu{j}{s}{i}{s} \nonumber  \\
&= \vddduud{j}{k}{i}{m}{k}{m} + c \vdddu{s}{j}{i}{s}  - c \vdudd{s}{s}{i}{j}  + c \vdddu{j}{i}{s}{s}  - c \vddud{j}{s}{s}{i}  \nonumber \\
 &= \vddduud{j}{k}{i}{m}{k}{m} + 
  c \vdddu{j}{i}{s}{s} + c^2 n v_{i j},
\end{align}
so commuting derivatives again gives
\begin{align} \label{eqn:loc-com-1}
\vddduud{j}{k}{i}{m}{k}{m} &= 
\vdddudu{j}{k}{i}{k}{m}{m}  + (\riemdddu{m}{k}{j}{s} \vdud{s}{k}{i}  + \riemdudu{m}{k}{k}{s} v_{{j}{s},{i}} + \riemdddu{m}{k}{i}{s} \vdud{j}{k}{s})_{,}^{\phantom{,}m} \nonumber \\
&= \vdddudu{j}{k}{i}{k}{m}{m}  + c \vdddu{s}{j}{i}{s}  - c \vdudd{s}{s}{i}{j}  + c(n-1) \vdddu{j}{s}{i}{s}  +  c \vdddu{j}{i}{s}{s} - c \vdudd{j}{s}{s}{i} \nonumber \\
&= \vdddudu{j}{k}{i}{k}{m}{m}  + cn \vdddu{j}{s}{i}{s}  +  c \vdddu{j}{i}{s}{s}  - c \vdudd{j}{s}{s}{i}.
\end{align}
Substituting 
\begin{align*}
\vdddu{j}{s}{i}{s}  = \vddud{j}{s}{s}{i} + \riemdudu{i}{s}{j}{p}  v_{ps} + \riemdudu{i}{s}{s}{p} v_{jp}  = \vddud{j}{s}{s}{i} + cn v_{ij}.
\end{align*}
into \eqref{eqn:loc-com-1}, we obtain
\begin{align*}
\vddduud{j}{k}{i}{m}{k}{m} &= \vdddudu{j}{k}{i}{k}{m}{m} + c(n-1) \vddud{j}{s}{s}{i} + c^2 n^2 v_{ij}  +  c \vdddu{i}{j}{s}{s},
\end{align*}
so \eqref{eqn:loc-com-0} becomes
\begin{align}\label{eqn:loc-com-2}
\vdddudu{j}{k}{i}{m}{m}{k} &= \vdddudu{j}{k}{i}{k}{m}{m} + c(n-1) \vddud{j}{s}{s}{i}  + c^2 (n^2+n) v_{ij}  +  2 c \vdddu{i}{j}{s}{s}.
\end{align}
Next, 
\begin{align*}
\vdddudu{j}{k}{i}{k}{m}{m} = \vdduddu{j}{k}{k}{i}{m}{m} + c n \vdddu{j}{i}{m}{m},
\end{align*}
which transforms \eqref{eqn:loc-com-2} to 
\begin{align} \label{eqn:loc-com-3}
\vdddudu{j}{k}{i}{m}{m}{k} &=\vdduddu{j}{k}{k}{i}{m}{m} + c (n+2) \vdddu{i}{j}{m}{m} + c(n-1) \vddud{j}{s}{s}{i} + c^2 (n^2+n) v_{ij}.
\end{align}
Similarly, 
\begin{align*}
\vdduddu{j}{k}{k}{i}{m}{m}   \vdduddu{j}{k}{k}{m}{i}{m} + c \vddud{i}{m}{m}{j}  - c \vdudu{p}{m}{m}{p}  h_{ij},   
\end{align*}
which replaces \eqref{eqn:loc-com-3} by 
\begin{align*}
\vdddudu{j}{k}{i}{m}{m}{k} = \, & \vdduddu{j}{k}{k}{m}{i}{m} + c \vddud{i}{m}{m}{j} - c \vdudu{p}{m}{m}{p}  h_{ij} + c (n+2) \vdddu{i}{j}{m}{m}  
\nonumber \\ & + c(n-1) \vddud{j}{s}{s}{i}  + c^2 (n^2+n) v_{ij}.
\end{align*}
Commuting derivatives one final time, 
\begin{align*}
\vdduddu{j}{k}{k}{m}{i}{m} 
 &= \vddudud{j}{k}{k}{m}{m}{i} +  c \vddud{i}{m}{m}{j}  - c \vdudu{p}{m}{m}{p}  h_{ij} + c(n-1) \vddud{j}{k}{k}{i},
\end{align*}
which yields \eqref{eqn:commut-trfree}.
\end{proof}

\end{document}